\documentclass[11pt]{article}
\usepackage[margin=1in]{geometry}
\usepackage{amsmath,amssymb,amsthm,mathtools,bm}
\usepackage{booktabs,longtable,array}
\usepackage{microtype}
\usepackage{enumitem}
\usepackage{hyperref}
\usepackage{xcolor}
\usepackage{graphicx}
\usepackage{authblk}
\usepackage{algorithm}
\usepackage{algpseudocode}
\usepackage[numbers,sort&compress]{natbib}
\hypersetup{colorlinks=true,linkcolor=blue,citecolor=blue,urlcolor=blue}

\newtheorem{definition}{Definition}[section]
\newtheorem{proposition}[definition]{Proposition}
\newtheorem{theorem}[definition]{Theorem}
\newtheorem{lemma}[definition]{Lemma}
\newtheorem{corollary}[definition]{Corollary}
\newtheorem{remark}[definition]{Remark}
\newtheorem{assumption}[definition]{Assumption}
\newtheorem{principle}[definition]{Design Principle}

\newcommand{\Rel}{\mathbf{Rel}}
\newcommand{\Vect}{\mathbf{Vect}}

\newcommand{\Cl}{\operatorname{Cl}}

\newcommand{\Path}{\operatorname{Path}}
\newcommand{\Rep}{\operatorname{Rep}}
\newcommand{\Hom}{\operatorname{Hom}}
\newcommand{\Ext}{\operatorname{Ext}}
\newcommand{\im}{\operatorname{im}}
\newcommand{\id}{\operatorname{id}}

\newcommand{\Score}{\operatorname{Score}}

\newcommand{\Prov}{\operatorname{Prov}}
\newcommand{\At}{\operatorname{At}}
\newcommand{\Reach}{\operatorname{Reach}}
\newcommand{\Iso}{\operatorname{Iso}}
\newcommand{\Energy}{\mathcal{E}}
\newcommand{\KG}{\mathcal{K}}
\newcommand{\T}{\mathcal{T}}
\newcommand{\A}{\mathcal{A}}
\newcommand{\Lcal}{\mathcal{L}}
\newcommand{\F}{\mathcal{F}}
\newcommand{\M}{\mathcal{M}}

\title{PAGR: Proof-Carrying Algebraic-Geometric Retrieval:\\
A Quiver-, Provenance-, and Sheaf-Theoretic Framework for Grounded LLM Retrieval}
\author[1]{Xingting Wang}
\author[2]{Min Wu}
\affil[1]{Department of Mathematics, Louisiana State University, Baton Rouge, LA 70803, USA \quad \texttt{xingtingwang@lsu.edu}}
\affil[2]{Independent Researcher \quad \texttt{min.wu.2020@gmail.com}}

\date{September 2026}

\begin{document}
\maketitle

\begin{abstract}
Retrieval-augmented generation (RAG) is usually formulated as a statistical information-retrieval problem: a query is embedded, nearby text is retrieved, and a language model is conditioned on the result. Graph-based RAG adds relational structure, but the mathematical status of that structure is often underspecified. In particular, three distinct questions are frequently conflated: which statements are certified as knowledge, which latent representations are useful for retrieval, and which multi-hop compositions are semantically admissible.

We propose a use-case-agnostic framework called \emph{Proof-Carrying Algebraic-Geometric Retrieval} (PAGR), which separates and then couples these questions through four mathematical layers. The \emph{symbolic layer} is a finite many-sorted relational theory generated by a typed quiver, path equations, and positive Horn inclusions. The \emph{algebraic representation layer} maps the quiver into finite-dimensional inner-product spaces, so relation types act as linear operators and compositional constraints become operator identities or regularizers. The \emph{geometric layer} equips entity types with metric or mixed-geometry representation spaces for semantic seeding and inductive retrieval. The \emph{sheaf layer} encodes local compatibility and supplies a local-to-global consistency energy. A fifth, orthogonal mechanism---semiring provenance---attaches machine-checkable derivation certificates to symbolic inferences.

The central architectural constraint is an epistemic separation principle: learned geometry may rank, propose, and organize evidence, but it may not promote a hypothesis into certified ground truth. Certification is reserved for source-attested base facts and derivations admitted by a sound symbolic theory. We formalize this distinction and prove four results: a \emph{non-interference} theorem stating that the certification boundary---the criterion by which a retrieved fact qualifies as certified---is invariant under arbitrary replacement of every learned component, even though which certified facts are retrieved is not, and from which soundness relative to an intended model follows as a corollary; a conditional \emph{completeness} bound that characterizes exactly which certified facts geometric seeding can fail to reach; a gauge-invariance theorem identifying the isometry group---rather than the full base-change group---as the symmetry that retrieval scores actually respect; and a proposition showing that zero sheaf energy on a retrieved subgraph forces all declared path equations to hold along it, which couples the sheaf and representation layers instead of leaving them parallel. Three further groups of results sharpen and delimit the construction. Sheaf cohomology separates whether a particular retrieved configuration is consistent from whether any consistent configuration exists at all---the latter measured by the dimension of global sections, and assembled across overlapping neighborhoods by a Mayer--Vietoris gluing theorem. Bounded bisimulation yields a minimal index that decides admissible-path realizability exactly, while provenance decides which base facts are redundant. Finally, a scope theorem gives a structural criterion for when a schema leaves the finite and tame representation regimes, delimiting which classical tools the algebraic layer may legitimately invoke. We then define a hybrid retrieval operator combining metric seeding, admissible path expansion, provenance, representation residuals, sheaf consistency, and graph diffusion. The result is a mathematical blueprint for LLM retrieval systems in which geometry improves recall and organization while algebra and proof constrain what may be treated as ground truth.
\end{abstract}

\section{Introduction}

Large language models possess broad parametric knowledge but provide weak guarantees about factuality, attribution, freshness, or the provenance of individual claims. Retrieval-augmented generation \citep{lewis2020} addresses this by conditioning generation on external evidence. Modern graph-based variants extend retrieval from isolated chunks to relational neighborhoods, subgraphs, and multi-hop structures \citep{edge2024,han2025,hu2024,zhu2025,gutierrez2024}. This is a substantial improvement for questions whose answers depend on relational context. Yet it leaves a foundational issue unresolved: a graph edge can mean at least three different things.

First, an edge may be a \emph{certified assertion}: a source explicitly states a relation. Second, an edge may be a \emph{logical consequence}: it follows from certified assertions under a declared rule system. Third, an edge may be a \emph{statistical hypothesis}: a learned model predicts that the relation is likely. These objects are useful for different purposes, but treating them as interchangeable makes the phrase ``knowledge graph'' mathematically ambiguous and makes ``grounded generation'' difficult to define.

The same ambiguity appears in geometry. Knowledge-graph embedding (KGE) methods associate entities and relations with vectors, rotations, projections, tensors, or transformations and can model symmetry, inversion, hierarchy, and composition \citep{bordes2013,trouillon2016,sun2019,li2024golde,yusupov2025}. These representations are extremely useful for retrieval and link prediction. But a small embedding distance is not a proof that a relation is true. Conversely, a symbolic graph with no geometry can be exact but brittle: it may fail to retrieve semantically relevant regions when the user's language does not match the graph's vocabulary.

This paper develops a mathematical framework in which symbolic and geometric knowledge coexist without being conflated. The framework is deliberately independent of any particular domain. It can represent scientific claims, software dependencies, corporate knowledge, legal authorities, biomedical facts, product catalogs, educational concepts, or any other corpus for which typed relations and source evidence exist.

Our proposal uses several established theories, but assigns each a sharply delimited role:
\begin{enumerate}[leftmargin=*]
    \item \textbf{Typed quivers and path algebras} represent the grammar of relation composition.
    \item \textbf{Positive Horn rules and relational fixed points} define certified symbolic closure.
    \item \textbf{Quiver representations} assign vector spaces to types and linear maps to relation generators, making compositional geometry explicit.
    \item \textbf{Metric and mixed-geometry embeddings} support semantic seeding and inductive ranking.
    \item \textbf{Cellular sheaves} encode local compatibility and quantify whether local representations glue coherently.
    \item \textbf{Semiring provenance} records derivations independently of any one numerical confidence convention.
    \item \textbf{Proof-carrying admission} requires a checkable witness before an inferred statement may enter the certified retrieval context.
\end{enumerate}

None of these ingredients is new in isolation. Quiver representations are classical, have been used to analyze neural architectures \citep{armenta2020,ganev2022}, and admit geometric moduli-space formulations \citep{jeffreys2021}. Cellular sheaves have recently been used to generalize graph neural networks and knowledge-graph embeddings \citep{hansen2020,gebhart2023,braithwaite2024}. Provenance semirings provide a general algebra of derivations \citep{green2007,grahne2020}. The contribution here is a unified epistemic architecture for retrieval: symbolic truth, learned geometry, relational algebra, local consistency, and generative synthesis are explicitly separated and connected by typed interfaces.

\paragraph{Core thesis.} A mathematically trustworthy retrieval system should use geometry to answer \emph{where should we look?}, algebra and logic to answer \emph{what compositions are admissible?}, provenance to answer \emph{why is this conclusion present?}, and a generator only to answer \emph{how should the certified context be communicated?}

\section{Related Work and Positioning}

\subsection{Graph-based retrieval augmented generation}
GraphRAG systems retrieve or construct graph-structured context rather than treating text chunks as independent units. Recent frameworks build LLM-derived entity graphs and summarize communities for global queries \citep{edge2024}, retrieve textual subgraphs \citep{hu2024}, expand seed chunks using graph neighborhoods \citep{zhu2025}, or run personalized PageRank from query-linked seeds over an extracted knowledge graph \citep{gutierrez2024}; \citet{han2025} survey the space. Our Stage V is deliberately the same diffusion primitive used by \citet{gutierrez2024}, which makes the comparison sharp: the difference is not the diffusion but what the diffused subgraph is permitted to license. We adopt the high-level GraphRAG motivation but focus on a different problem: the semantics of graph edges and the conditions under which graph-derived information is admissible as ground truth.

\subsection{Knowledge graph embeddings and relational geometry}
Latent KGE models encode relational regularities using algebraic transformations. TransE models relations as translations \citep{bordes2013}; ComplEx uses complex-valued factorization and can capture symmetric and antisymmetric relations \citep{trouillon2016}; RotatE models relations as complex rotations and supports inversion and composition patterns \citep{sun2019}. More recent approaches use hyperbolic geometry for hierarchy \citep{chami2020}, higher-dimensional orthogonal parameterizations, or mixed Euclidean/hyperbolic geometry \citep{li2024golde,yusupov2025}. Theoretical work also shows that not every relation is equally representable under simple Euclidean embedding families \citep{bhattacharjee2020}. Our framework therefore treats the choice of geometry as a model class, not as an ontological commitment.

\subsection{Quivers, representations, and machine learning}
A quiver is a directed multigraph; a representation assigns a vector space to each vertex and a linear map to each arrow \citep{derksen2017,schiffler2014}. Neural networks themselves can be described using quiver representations augmented with nonlinearities \citep{armenta2020,ganev2022}. We instead apply quiver representation theory to the \emph{knowledge schema}: entity types are vertices, relation types are arrows, and learned relation operators form a representation of that schema.

\subsection{Sheaf-theoretic knowledge representations}
Cellular sheaves attach vector spaces to cells and restriction maps to incidences; the spectral theory of the associated sheaf Laplacian is developed by \citet{hansen2019spectral}. Sheaf neural networks replace ordinary graph diffusion by a sheaf Laplacian, allowing edge-dependent and asymmetric compatibility structure \citep{hansen2020,bodnar2022,braithwaite2024}. Knowledge Sheaves formulates KGE as approximate global-section finding and uses schema constraints to model composite relations \citep{gebhart2023}. Section~\ref{sec:sheaf} makes the relationship to our representation layer explicit: the quiver representation is recovered as the special case in which restriction maps are \emph{tied by relation type} rather than free per edge, and we show that this tying makes sheaf consistency imply representation-layer compositional consistency. We keep the sheaf layer epistemically subordinate to certified symbolic facts: sheaf consistency is evidence of representational coherence, not itself a proof of truth.

\subsection{Provenance and proof-carrying computation}
Provenance semirings annotate base facts and propagate annotations through query derivations \citep{green2007}; see \citet{abiteboul1995} for the underlying Datalog fixed-point semantics. Extensions address regular path queries and recursive Datalog semantics \citep{grahne2020,bourgeois2022,khamis2021}, and $m$-semirings equip the framework with a monus operator supporting deletion \citep{geerts2010}. Our admission principle is conceptually related to proof-carrying code \citep{necula1997}: a consumer accepts an object only together with a mechanically checkable witness that it satisfies a declared policy. Here the object is an inferred fact rather than executable code. The invariance form of our central theorem (Section~\ref{sec:noninterference}) is a non-interference statement in the sense of \citet{goguen1982}, with epistemic status playing the role of a security level.

\subsection{Positioning against concurrent neuro-symbolic retrieval}
\label{sec:concurrent}
Two concurrent lines of work share our vocabulary and must be distinguished explicitly. \citet{koomullil2026} propose proof-carrying certificates for LLM pipelines with a Lean~4 trust boundary, verifying the deterministic computation \emph{surrounding} a model, including bilattice grounding and embedding-sensitivity certificates. That work is about kernel-checkable certificates for pipeline steps; ours is about the algebraic and geometric structure of the knowledge object being retrieved, and in particular about which multi-hop compositions are admissible. The two are complementary: a PAGR certificate $\delta$ is exactly the kind of artifact a Lean-style checker would consume. \citet{gann2026} synthesize attributable Prolog modules from retrieved chunks and obtain execution traces linking reasoning steps to sources; this realizes the provenance and certificate layers empirically, but does not separate the representation and consistency layers, does not type the composition grammar with a schema quiver, and does not address the geometric seeding stage. We regard NeSy-RAG as the closest empirical instantiation of Stages~II--III and a natural baseline for the hypotheses in Section~\ref{sec:hypotheses}. Finally, citation-level generation evaluation \citep{gao2023} measures whether a generator's claims are supported by its context; our Design Principles constrain what enters that context in the first place, so the two are orthogonal and should be reported together.

\section{Mathematical Preliminaries}

\subsection{Typed quivers and the free path category}

\begin{definition}[Typed schema quiver]
A \emph{schema quiver} is a finite directed multigraph
\[
Q=(Q_0,Q_1,s,t),
\]
where $Q_0$ is a finite set of entity sorts, $Q_1$ is a finite set of primitive relation types, and $s,t:Q_1\to Q_0$ assign the source and target sort of each relation.
\end{definition}

A path $p=a_n\cdots a_1$ is a composable sequence of arrows with $t(a_i)=s(a_{i+1})$. The length-zero path at $i\in Q_0$ is denoted $e_i$. Paths form the free category $\Path(Q)$ under concatenation.

Fix a field $k$, usually $\mathbb{R}$ or $\mathbb{C}$.

\begin{definition}[Path algebra]
The \emph{path algebra} $kQ$ is the $k$-vector space with basis all finite paths in $Q$, equipped with bilinear multiplication given by concatenation when paths are composable and zero otherwise.
\end{definition}

The idempotents $e_i$ satisfy $e_i^2=e_i$ and $e_ie_j=0$ for $i\neq j$. They encode the type decomposition algebraically.

\subsection{Equations versus inclusions}
Two kinds of semantic constraint must be separated.

\paragraph{Path equations.} If two compositions are intended to denote the same operation, we may impose $p=q$. Let $J\triangleleft kQ$ be the ideal generated by differences $p-q$. Then
\[
\A_Q=kQ/J
\]
is the corresponding quotient of the path algebra. This is the correct use of quotient/path-algebra machinery.

\begin{remark}[Admissibility]
\label{rem:admissible}
The classical theory of \emph{bound quiver algebras} requires $J$ to be admissible, i.e.\ $R^m\subseteq J\subseteq R^2$ for some $m$, where $R\triangleleft kQ$ is the arrow ideal \citep{derksen2017,schiffler2014}. Admissibility guarantees, among other things, that $kQ/J$ is finite dimensional and that $Q$ is recovered as its Gabriel quiver. Some schema constraints of practical interest violate this condition. In particular the inverse-pair equations $ba=e_i$, $ab=e_j$ of Section~\ref{sec:constraint-reg} involve length-zero paths and so generate a non-admissible ideal; $kQ/J$ is then a groupoid-algebra-like quotient rather than a bound quiver algebra. Nothing in this paper requires admissibility---we only ever use $\A_Q$ through the factorization criterion of Proposition~\ref{prop:factorization}, which holds for an arbitrary two-sided ideal---but the classification results of representation theory should not be invoked unless $J$ is admissible. We therefore write ``quotient path algebra'' rather than ``bound path algebra'' throughout.
\end{remark}

\paragraph{Relational inclusions.} A rule such as
\[
R_1;R_2\subseteq R_3
\]
does \emph{not} assert equality of paths. It is a positive Horn implication. Such rules are handled in a relational semantics rather than by quotienting $kQ$.

\begin{principle}[Equation--inclusion separation]
\label{prin:eq-incl}
Use path-algebra quotients only for identities. Use Horn closure for one-way semantic implications. Conflating the two changes the meaning of the knowledge system.
\end{principle}

\subsection{Relational structures}
For each sort $i\in Q_0$, let $P_i$ be a finite set of entities. For an arrow $a:i\to j$, a concrete relation is
\[
E_a\subseteq P_i\times P_j.
\]
Composition is ordinary relational composition. The collection of sets and relations lives in the locally ordered category $\Rel$, where hom-sets are ordered by inclusion.

\section{The PAGR Framework}

\begin{definition}[PAGR knowledge system]
\label{def:pagr}
A \emph{PAGR knowledge system} is a tuple
\[
\KG=(Q,P,E,\T,J,\Phi,\rho,\F,\eta),
\]
where:
\begin{enumerate}[leftmargin=*]
    \item $Q$ is a finite schema quiver;
    \item $P=\bigsqcup_{i\in Q_0}P_i$ is the finite typed entity universe;
    \item $E=\{E_a\}_{a\in Q_1}$ is the extensional base relation set;
    \item $\T$ is a finite set of positive Horn relation rules;
    \item $J$ is an optional ideal of path equations;
    \item $\Phi=\{\phi_i:P_i\to \M_i\}$ embeds entities into metric or Riemannian spaces;
    \item $\rho$ is a finite-dimensional representation of the schema quiver;
    \item $\F$ is a cellular sheaf encoding local compatibility data;
    \item $\eta$ maps certified base facts to provenance generators and source evidence.
\end{enumerate}
\end{definition}

This definition intentionally contains both discrete and continuous objects. Their responsibilities differ. $E,\T,J$ define certified symbolic semantics. $\Phi,\rho,\F$ define retrieval and consistency geometry. $\eta$ supplies evidence and derivation identity.

\section{Certified Symbolic Semantics}
\label{sec:certified}

\subsection{The extensional base and closure}
Let $\mathfrak{R}$ be the finite lattice of all typed relation families over $P$. A positive Horn rule has the form
\[
R_{a_1};\cdots;R_{a_m}\subseteq R_b,
\]
possibly with repeated variables and type-compatible joins.

Define the immediate-consequence operator $T_{\T}:\mathfrak{R}\to\mathfrak{R}$ by
\[
T_{\T}(X)=X\cup\{\text{every tuple licensed by one application of a rule in }\T\text{ to }X\}.
\]
\begin{remark}[Inflationarity]
\label{rem:inflationary}
$T_{\T}$ is \emph{inflationary}: $X\subseteq T_{\T}(X)$ for every $X\in\mathfrak{R}$, since the definition retains all elements of $X$ in the union. This is what makes the ascending chain below well-defined; it is stated separately because monotonicity alone does not imply it.
\end{remark}

\begin{proposition}[Monotonicity]
$T_{\T}$ is monotone: if $X\subseteq Y$, then $T_{\T}(X)\subseteq T_{\T}(Y)$.
\end{proposition}
\begin{proof}
Every rule in $\T$ is positive: premises contain no negation. Any witness tuple satisfying a premise in $X$ also satisfies it in $Y$. Hence every consequence added from $X$ is also added from $Y$, and $X\subseteq Y$ handles the copied part.
\end{proof}

\begin{theorem}[Finite least closure]
\label{thm:closure}
For finite $P$, the sequence
\[
E\subseteq T_{\T}(E)\subseteq T_{\T}^{2}(E)\subseteq\cdots
\]
stabilizes after finitely many steps at a unique least fixed point of $T_{\T}$ above $E$,
\[
\Cl_{\T}(E)=\operatorname{lfp}(T_{\T}).
\]
\end{theorem}
\begin{proof}
The chain is ascending by Remark~\ref{rem:inflationary}, and $\mathfrak{R}$ is a finite lattice, so the chain stabilizes at some $N\le|\mathfrak{R}|$; write $F=T_{\T}^{N}(E)$, which satisfies $T_{\T}(F)=F$ and $E\subseteq F$. For minimality, let $Y$ be any fixed point with $E\subseteq Y$. By induction, $T_{\T}^{0}(E)=E\subseteq Y$, and if $T_{\T}^{n}(E)\subseteq Y$ then monotonicity gives $T_{\T}^{n+1}(E)\subseteq T_{\T}(Y)=Y$. Hence $F\subseteq Y$. (Equivalently, apply Knaster--Tarski to the monotone $T_{\T}$ on the finite lattice $\{X\in\mathfrak{R}:E\subseteq X\}$.)
\end{proof}

Thus certified inference is deterministic and terminating on a finite knowledge base without appealing to learned representations; this is the standard fixed-point semantics of positive Datalog \citep{abiteboul1995}.

\subsection{Model-relative soundness}

Let $\mathfrak{M}$ be a typed relational structure over the same signature, and write
\[
\At(\mathfrak{M})=\{f : f\text{ is a typed atomic statement and }\mathfrak{M}\models f\}
\]
for the set of atoms it satisfies. Comparing a derived fact set to $\At(\mathfrak{M})$, rather than to $\mathfrak{M}$ itself, keeps the containment statements below type-correct.

\begin{assumption}[Source correctness]
There exists an intended relational model $\mathfrak{M}$ such that $E\subseteq\At(\mathfrak{M})$.
\end{assumption}

\begin{assumption}[Rule soundness]
Every rule in $\T$ is valid in $\mathfrak{M}$.
\end{assumption}

\begin{theorem}[Soundness of certified closure]
\label{thm:closure-sound}
Under Source Correctness and Rule Soundness,
\[
\Cl_{\T}(E)\subseteq \At(\mathfrak{M}).
\]
\end{theorem}
\begin{proof}
Induct on the stage $n$ of the fixed-point construction of Theorem~\ref{thm:closure}. For $n=0$, $T_{\T}^{0}(E)=E\subseteq\At(\mathfrak{M})$ by Source Correctness. Suppose $T_{\T}^{n}(E)\subseteq\At(\mathfrak{M})$. A fact in $T_{\T}^{n+1}(E)\setminus T_{\T}^{n}(E)$ is licensed by one application of a rule of $\T$ whose premises lie in $T_{\T}^{n}(E)$ and hence hold in $\mathfrak{M}$; Rule Soundness makes the conclusion hold in $\mathfrak{M}$. Since the chain stabilizes at $\Cl_{\T}(E)$, the claim follows.
\end{proof}

The theorem is conditional, as it must be: mathematics cannot establish that a source is factually correct merely from the fact that it was ingested. What the theorem establishes is that the inference engine introduces no conclusions outside the declared semantics.

\section{Epistemic Stratification and Ground Truth}

We now formalize the distinction between certified facts and model hypotheses.

\begin{definition}[Epistemic strata]
For each typed atomic statement $f$, assign one of three statuses:
\[
\mathsf{status}(f)\in\{\mathsf{base},\mathsf{derived},\mathsf{hypothesis}\}.
\]
\begin{itemize}
    \item $\mathsf{base}$: directly supported by admitted source evidence;
    \item $\mathsf{derived}$: belongs to $\Cl_{\T}(E)$ and has a valid derivation certificate;
    \item $\mathsf{hypothesis}$: proposed by geometry, embedding, an LLM, a link predictor, or any process not licensed by $\T$.
\end{itemize}
\end{definition}

\begin{definition}[Certified ground-truth set]
\label{def:gstar}
The internal certified knowledge set is the closure $\Cl_{\T}(E)$ equipped with the status labelling
\[
G^{\star}=\underbrace{E}_{\mathsf{base}}\;\sqcup\;\underbrace{\bigl(\Cl_{\T}(E)\setminus E\bigr)}_{\mathsf{derived}},
\]
subject to the additional requirement that each $\mathsf{derived}$ member carry a valid proof object.
\end{definition}

\begin{remark}
As a set, $G^{\star}=\Cl_{\T}(E)$, since $E\subseteq\Cl_{\T}(E)$; the content of Definition~\ref{def:gstar} is the \emph{partition} into strata and the proof-object requirement on the second block, not the union. We write it as a disjoint union to emphasize that the two blocks are governed by different admission mechanisms: evidence for $\mathsf{base}$, derivation for $\mathsf{derived}$. The complement $\mathfrak{R}\setminus G^{\star}$ carries status $\mathsf{hypothesis}$.
\end{remark}

The term ``ground truth'' throughout means \emph{certified relative to an explicit evidence protocol and a declared rule set}, not metaphysical truth. This distinction is essential in dynamic or contested domains, and should be stated prominently in any system description.

\begin{principle}[No geometric promotion]
\label{prin:no-promotion}
No numerical score, embedding similarity, neural prediction, or LLM assertion can change a fact's status from $\mathsf{hypothesis}$ to $\mathsf{base}$ or $\mathsf{derived}$. Promotion requires new admitted evidence or a valid symbolic derivation.
\end{principle}

This principle is an architectural invariant rather than a statistical calibration rule.

\subsection{Extraction is not attestation}
\label{sec:extraction}

The principle above invites an immediate objection. In any realistic deployment the base set $E$ is populated by running an information-extraction model---frequently an LLM---over a text corpus. If LLM assertions cannot become base facts, how does $E$ acquire any members at all? Confronting this squarely is essential, because it is where most implementations of ``grounded'' retrieval quietly leak.

The resolution is to distinguish two different propositions. Let $\sigma$ be a source span.
\begin{itemize}[leftmargin=*]
    \item \textbf{Attestation:} \emph{the span $\sigma$ asserts $f$.} This is a claim about the text, checkable in principle by re-reading $\sigma$.
    \item \textbf{Extraction:} \emph{a model believes that $\sigma$ asserts $f$.} This is a statistical claim about the model.
\end{itemize}
An extractor produces the second; the admission predicate must decide the first. The framework is therefore consistent with LLM-based ingestion provided that extraction output is treated as a $\mathsf{hypothesis}$ about attestation, and that the admission predicate applies an independent check---span-grounded verification, a second extractor with disagreement routing, human curation, or a schema-and-type validity test---before the pair $(f,\epsilon_f)$ enters $E$.

\begin{principle}[Attestation gate]
\label{prin:attestation}
An extractor may propose $(f,\sigma)$ but may not admit it. Admission is the act of an evidence policy $\operatorname{Admit}$ that is independent of the extractor and that binds $f$ to a stable, re-checkable representation of $\sigma$.
\end{principle}

This weakens no theorem below. Soundness is stated relative to Source Correctness, and Source Correctness is now explicitly a property of the admission policy rather than of the extractor. The practical consequence is a measurable quantity: the \emph{attestation error rate} of the admission policy, which should be reported as a first-class system statistic alongside retrieval metrics.

\section{Provenance as a Derivation Algebra}

Assign each certified base fact $e\in E$ an indeterminate $x_e$. For positive relational queries, provenance can be represented in the polynomial semiring $\mathbb{N}[X]$ \citep{green2007}. Products represent joint use of premises and sums represent alternative derivations.

\begin{definition}[Proof provenance]
For a derived fact $f$, let
\[
\Prov(f)=\sum_{\delta\in D(f)}\prod_{e\in \mathrm{leaves}(\delta)}x_e,
\]
where $D(f)$ is the set of admitted finite derivations of $f$.
\end{definition}

For recursive rules, care is needed because derivation families may be infinite; semiring semantics for recursive Datalog require additional algebraic conditions or finite proof representations \citep{bourgeois2022,khamis2021}. In a bounded-retrieval implementation one may instead retain a finite derivation DAG up to depth $h$.

\begin{definition}[Proof-carrying fact]
A proof-carrying fact is a pair $(f,\delta)$ where $f$ is an atomic statement and $\delta$ is a finite derivation tree whose leaves resolve to admitted source evidence and whose internal nodes instantiate rules in $\T$.
\end{definition}

\begin{proposition}[Cheap verification]
For fixed rule arity and an explicit finite derivation tree $\delta$, verification of $(f,\delta)$ is linear in the size of $\delta$ plus the cost of resolving its leaves.
\end{proposition}

This is the retrieval analogue of the proof-carrying design pattern \citep{necula1997}: inference may be computationally expensive, but consumption requires only proof checking.

\section{Quiver Representations as Relational Geometry}

\subsection{Representation of the schema}

\begin{definition}[Linear representation]
A finite-dimensional representation $\rho$ of $Q$ assigns a finite-dimensional inner-product space $V_i$ to each $i\in Q_0$ and a linear map
\[
A_a:V_{s(a)}\to V_{t(a)}
\]
to every arrow $a\in Q_1$.
\end{definition}

For a path $p=a_n\cdots a_1$, define
\[
A_p=A_{a_n}\cdots A_{a_1}.
\]
Thus relation composition is literal operator composition.

If $J$ is generated by path equations $p-q$, the representation factors through $kQ/J$ precisely when $A_p=A_q$ for every generator of $J$.

\begin{proposition}[Factorization criterion]
\label{prop:factorization}
Extend $\rho$ linearly to $kQ$ by $r=\sum_i \lambda_i p_i \mapsto A_r=\sum_i\lambda_i A_{p_i}$. Then $\rho$ descends to a module over $\A_Q=kQ/J$ if and only if $A_r=0$ for every $r\in J$, and it suffices to check this on a generating set of $J$; equivalently, all declared path equations hold under $\rho$.
\end{proposition}
\begin{proof}
A representation of $Q$ is the same thing as a $kQ$-module, and a $kQ$-module descends to $kQ/J$ exactly when $J$ annihilates it. Since $\ker(\rho)$ is a two-sided ideal, $J\subseteq\ker(\rho)$ holds iff every generator of $J$ lies in $\ker(\rho)$.
\end{proof}

This gives the path algebra a concrete role in learning: exact path equations define hard constraints on the operators, while approximate equality defines structural regularizers. Let
\[
\mathcal{E}_J=\{(p,q): p-q\text{ is a declared generator of }J\}
\]
be the finite set of declared path-equation pairs. Then
\[
\Lcal_{\mathrm{alg}}=\sum_{(p,q)\in\mathcal{E}_J}\|A_p-A_q\|_F^2 .
\]

\subsection{Entity states}
\label{sec:entity-states}
For an entity $x\in P_i$, let $z_x\in V_i$ denote a learned or encoded state. A relational residual for $a:i\to j$ is
\[
r_a(x,y)=\|A_a z_x-z_y\|_{V_j}.
\]
Small residual means the representation predicts compatibility with relation $a$. It does \emph{not} certify $(x,a,y)$ as true.

\subsection{Gauge symmetry}
Choice of basis in each $V_i$ is not semantically meaningful.

\begin{definition}[Unitary gauge transformation]
For each $i\in Q_0$, choose $U_i\in U(V_i)$. Transform
\[
z_x' = U_i z_x,\qquad A_a'=U_{t(a)}A_aU_{s(a)}^{-1}.
\]
\end{definition}

\begin{theorem}[Gauge invariance of relational residuals]
\label{thm:gauge}
For unitary gauge transformations,
\[
\|A_a' z_x'-z_y'\|=\|A_a z_x-z_y\|.
\]
More generally, the same holds for every path residual $\|A_p z_x-z_y\|$.
\end{theorem}
\begin{proof}
For an arrow $a:i\to j$,
\[
A_a'z_x'-z_y'=U_jA_aU_i^{-1}U_i z_x-U_jz_y=U_j(A_az_x-z_y).
\]
Unitary maps preserve norm. For a path $p=a_n\cdots a_1$ from $i$ to $j$ the operators telescope, $A_p'=U_{j}A_pU_{i}^{-1}$, and the same computation applies.
\end{proof}

\begin{principle}[Gauge-invariant retrieval]
A retrieval criterion defined on a quiver representation should depend only on quantities invariant under the declared gauge group, not on arbitrary coordinate choices.
\end{principle}

This principle is useful when different encoders, checkpoints, or distributed components choose different internal bases.

\subsection{Representation varieties: which group is the gauge group?}
\label{sec:moduli}
For a dimension vector $\bm d=(d_i)_{i\in Q_0}$,
\[
\Rep(Q,\bm d)=\prod_{a:i\to j}\Hom(k^{d_i},k^{d_j})
\]
is the affine representation space. The base-change group
\[
G_{\bm d}=\prod_{i\in Q_0}GL(d_i)
\]
acts by simultaneous change of basis, and two points in the same $G_{\bm d}$-orbit are isomorphic as representations. Learned relation operators therefore live in a moduli problem, a viewpoint that has appeared in geometric analyses of neural networks and quiver varieties \citep{jeffreys2021,armenta2020,ganev2022}.

This orbit equivalence must not be conflated with the invariance established in Theorem~\ref{thm:gauge}; the distinction has practical consequences for system design.

\begin{remark}[The isometry subgroup, not the full base-change group]
\label{rem:gauge-gap}
Theorem~\ref{thm:gauge} holds for unitary $U_i$ only. If $g=(g_i)\in G_{\bm d}$ is a general invertible base change, then
\[
\|A_a'z_x'-z_y'\|=\|g_{j}(A_az_x-z_y)\|,
\]
which differs from $\|A_az_x-z_y\|$ unless $g_j$ is an isometry. Consequently two representations that are \emph{isomorphic}---the same point of the moduli space $\Rep(Q,\bm d)/\!\!/G_{\bm d}$---can induce different retrieval rankings. Residual-based retrieval is invariant under
\[
\Iso(\bm d)=\prod_{i\in Q_0}U(V_i)\;\subsetneq\;G_{\bm d},
\]
and the quotient by $G_{\bm d}$ is strictly coarser than the equivalence that retrieval respects.
\end{remark}

Three responses are available, and a system should state which it adopts. \emph{(i)} Fix the inner products and declare $\Iso(\bm d)$ to be the gauge group, treating each $V_i$ as an inner-product space rather than a bare vector space; the moduli-theoretic object is then a quotient by a compact group and Theorem~\ref{thm:gauge} is exactly right. \emph{(ii)} Parameterize relation operators inside the orthogonal or unitary group by construction, as orthogonally parameterized KGE models do \citep{li2024golde}; then $\Iso$-invariance is structural. \emph{(iii)} Report scores that are genuinely $G_{\bm d}$-invariant---ranks, kernel and image dimensions, or residuals normalized by learned metrics that transform covariantly---accepting a weaker but coordinate-free criterion. Option~(i) is our default, and it is what Design Principle \ref{prin:gauge-log} in Appendix~\ref{app:invariants} asks implementations to log.

We do not claim that classification of indecomposable representations is needed for retrieval. The useful consequence is narrower: quiver representation theory supplies the correct object for typed linear relation operators and exposes their symmetries---including, as Remark~\ref{rem:gauge-gap} shows, the fact that the naive symmetry group is the wrong one.

\section{Metric and Mixed-Geometry Semantic Spaces}

The representation $\rho$ captures relational transformations. Semantic retrieval additionally requires a notion of proximity. For each type $i$, let $(\M_i,d_i)$ be a metric space and
\[
\phi_i:P_i\to \M_i
\]
an entity embedding (the maps $\Phi=\{\phi_i\}$ appear in the PAGR tuple of Definition~\ref{def:pagr}). A text query $q$ is mapped by type-specific \emph{query encoders} $\psi_i:\mathcal{Q}\to \M_i$, where $\mathcal{Q}$ is the space of queries; these are auxiliary inference-time maps and need not be part of the stored system tuple.

The geometry need not be uniformly Euclidean. Hierarchical relations often motivate hyperbolic components; cyclic or rotational patterns motivate spherical or complex geometry; modern KGE systems increasingly use richer or mixed geometries \citep{sun2019,li2024golde,yusupov2025}.

\begin{definition}[Product geometry]
Let
\[
\M_i=\M_i^{(1)}\times\cdots\times\M_i^{(m)}
\]
with component distances $d_i^{(r)}$. A weighted product metric is
\[
d_i^2(x,y)=\sum_{r=1}^{m}\lambda_r d_i^{(r)}(x_r,y_r)^2,
\qquad \lambda_r\ge 0.
\]
\end{definition}

This makes ``geometric knowledge graph'' precise: geometry is a family of metric representation spaces connected by typed relational operators, not merely a generic vector index.

\subsection{Cross-type score commensurability}
\label{sec:commensurable}

Type-specific geometry creates a difficulty that a single global embedding space does not have, and it must be resolved before any statement about top-$k$ retrieval is meaningful. A raw distance $d_i(\psi_i(q),\phi_i(x))$ in $\M_i$ and a raw distance $d_j(\psi_j(q),\phi_j(y))$ in $\M_j$ are numbers in incomparable units: the spaces may have different dimensions, different curvature, and different diameter. Selecting a top-$k$ set over the whole entity universe $P=\bigsqcup_i P_i$ by comparing such numbers directly is not well defined.

We therefore require an explicit calibration step.

\begin{definition}[Calibrated semantic score]
\label{def:calibrated}
For each type $i$ let $F_i$ be the cumulative distribution function of $d_i(\psi_i(q),\phi_i(x))$ over a reference population of entities of type $i$. The calibrated score is
\[
s_{\mathrm{geo}}(q,x)=\pi_i\cdot\bigl(1-F_i\bigl(d_i(\psi_i(q),\phi_i(x))\bigr)\bigr),
\qquad x\in P_i,
\]
where $\pi_i\ge0$ is a type prior with $\sum_i\pi_i=1$.
\end{definition}

Composing with $F_i$ maps every type into the common unit interval, so cross-type comparison ranks entities by within-type extremeness weighted by a declared type prior. Any strictly decreasing calibration would serve; the point is that some choice must be declared, since the choice is a modeling commitment and not a normalization detail. All subsequent statements about $S_k(q)$ refer to $s_{\mathrm{geo}}$ in the sense of Definition~\ref{def:calibrated}.

\subsection{Inductive seeding for unseen entities}
\label{sec:inductive}

Transductive embeddings assign a vector only to entities present at training time, which is a poor fit for a knowledge base whose base set $E$ grows through the admission protocol of Section~\ref{sec:evidence}. Two mechanisms keep seeding usable for entities added after training, and neither disturbs the epistemic strata.

First, an entity of type $i$ that has a textual or otherwise featurized description can be embedded directly by $\phi_i$ whenever $\phi_i$ is a feature encoder rather than a lookup table; this is the inductive case in the strict sense. Second, a new entity with no usable features but with certified incident edges can be seeded structurally: subgraph-reasoning methods predict relations from the local structure around a candidate pair without entity-specific embeddings \citep{teru2020}, and relation-graph methods transfer relation representations to graphs with unseen entities and unseen relation vocabularies \citep{lee2023}.

Both mechanisms produce \emph{seeds}, not facts. An inductively embedded entity enters Stage~I exactly as any other entity does, and any relation predicted about it carries status $\mathsf{hypothesis}$ until admitted. The value of inductive capability in this architecture is therefore confined to recall, which is precisely where the framework intends learned components to operate.

\subsection{Top-$k$ stability}
Let $s_q(x)$ be a retrieval score, larger being better, obtained as in Definition~\ref{def:calibrated}.

\begin{theorem}[Margin stability of geometric seeding]
\label{thm:margin}
Assume a perturbation changes every entity score by at most $\varepsilon$:
\[
|s_q'(x)-s_q(x)|\le \varepsilon.
\]
Let $x_{(k)}$ and $x_{(k+1)}$ denote the $k$th and $(k+1)$st ranked entities under $s_q$. If
\[
s_q(x_{(k)})-s_q(x_{(k+1)})>2\varepsilon,
\]
then the top-$k$ seed set is unchanged under the perturbation.
\end{theorem}
\begin{proof}
The strict margin implies that the top-$k$ set is uniquely determined under $s_q$. Let $x$ be ranked in the top $k$ and $x'$ outside it, so that $s_q(x)\ge s_q(x_{(k)})$ and $s_q(x')\le s_q(x_{(k+1)})$. Then
\[
s_q'(x)\ \ge\ s_q(x)-\varepsilon\ \ge\ s_q(x_{(k)})-\varepsilon
\qquad\text{and}\qquad
s_q'(x')\ \le\ s_q(x')+\varepsilon\ \le\ s_q(x_{(k+1)})+\varepsilon .
\]
The hypothesis $s_q(x_{(k)})-s_q(x_{(k+1)})>2\varepsilon$ rearranges to $s_q(x_{(k)})-\varepsilon>s_q(x_{(k+1)})+\varepsilon$, hence $s_q'(x)>s_q'(x')$ for every such pair. No boundary crossing occurs and the top-$k$ set is unchanged.
\end{proof}

This gives a concrete robustness diagnostic for semantic retrieval: report the boundary margin, not only the nearest-neighbor scores. The scope of the guarantee should be noted. It concerns stability of the seed set under score perturbation, and says nothing about whether the seed set is correct: a systematically biased encoder is perfectly stable. Stability and recall are separate quantities; Section~\ref{sec:noninterference} characterizes the latter.

\section{Cellular Sheaves as Local-to-Global Consistency}
\label{sec:sheaf}

A quiver representation says how relation types act globally. A cellular sheaf can encode edge-specific compatibility between local states.

Let $G=(V,E_G)$ be an undirected support graph obtained from certified or candidate relations. A cellular sheaf $\F$ assigns a vector space $\F(v)$ to each vertex, a vector space $\F(e)$ to each edge, and restriction maps
\[
\F_{v\unlhd e}:\F(v)\to\F(e)
\]
for each incidence $v\unlhd e$.

A $0$-cochain is $x=(x_v)_{v\in V}\in C^0(G;\F)=\bigoplus_v\F(v)$. The coboundary on an oriented edge $e=(u,v)$ is
\[
(\delta x)_e=\F_{v\unlhd e}x_v-\F_{u\unlhd e}x_u.
\]
The sheaf Laplacian is
\[
L_{\F}=\delta^{\ast}\delta.
\]

\begin{definition}[Sheaf consistency energy]
\[
\Energy_{\F}(x)=\langle x,L_{\F}x\rangle=\|\delta x\|^2.
\]
\end{definition}

\begin{proposition}[Zero-energy consistency]
\label{prop:zero-energy}
\[
\Energy_{\F}(x)=0 \quad\Longleftrightarrow\quad \delta x=0.
\]
Hence zero-energy $0$-cochains are exactly global sections, $\ker L_{\F}=H^0(G;\F)$.
\end{proposition}
\begin{proof}
Since $L_{\F}=\delta^{\ast}\delta$,
\[
\langle x,L_{\F}x\rangle=\langle\delta x,\delta x\rangle=\|\delta x\|^2,
\]
which vanishes exactly when $\delta x=0$.
\end{proof}

This is the standard positive-semidefiniteness argument for the sheaf Laplacian \citep{hansen2019spectral} and we record it only to fix notation; the substantive content of this section is the coupling established below.

\subsection{The representation-induced sheaf}
\label{sec:induced-sheaf}

As defined so far, the sheaf layer and the representation layer are two separate assignments of vector spaces to combinatorial data, and the sheaf energy appears to be an independent add-on. This section shows they are not independent. Take the support graph $G$ to have one edge $e$ per certified relation instance, and let $\lambda(e)\in Q_1$ be its relation label.

\begin{definition}[Sheaf induced by a representation]
\label{def:induced-sheaf}
Given $\rho$, define $\F^{\rho}$ by $\F^{\rho}(v)=V_{\tau(v)}$ for a vertex of sort $\tau(v)$, $\F^{\rho}(e)=V_{t(\lambda(e))}$ for an edge, and, for $e$ an instance of $a:i\to j$ oriented from $u$ to $v$,
\[
\F^{\rho}_{u\unlhd e}=A_a,
\qquad
\F^{\rho}_{v\unlhd e}=\id_{V_j}.
\]
\end{definition}

Under this construction $(\delta x)_e = x_v - A_a x_u$, so the edgewise coboundary is exactly the relational residual $r_a$ of Section~\ref{sec:entity-states} and the sheaf energy is the sum of squared residuals over the subgraph. The general Knowledge Sheaves construction \citep{gebhart2023} allows both restriction maps to be free and relation-specific; Definition~\ref{def:induced-sheaf} is the special case in which restrictions are \emph{tied by relation type} through $\rho$. That tying is what makes the following proposition available.

\begin{proposition}[Sections certify path equations on their support]
\label{prop:section-path-eq}
Let $x$ be a global section of $\F^{\rho}$ on a subgraph $H\subseteq G$, i.e.\ $\Energy_{\F^{\rho}}(x|_H)=0$. Let $p=a_n\cdots a_1$ be a path in the schema and suppose $u_0,u_1,\dots,u_n$ is a sequence of vertices of $H$ with each $(u_{m-1},u_m)$ an edge of $H$ labelled $a_m$. Then
\[
A_p\,x_{u_0}=x_{u_n}.
\]
Consequently, if $p$ and $q$ are two such realized paths in $H$ with common endpoints $u_0$ and $w$, then
\[
(A_p-A_q)\,x_{u_0}=0,
\]
so every declared path equation of $\mathcal{E}_J$ realized in $H$ holds when applied to the section.
\end{proposition}
\begin{proof}
Zero energy gives $\delta x=0$, hence $x_{u_m}=A_{a_m}x_{u_{m-1}}$ for each edge of $H$ by Definition~\ref{def:induced-sheaf}. Induction on $m$ yields $x_{u_m}=A_{a_m}\cdots A_{a_1}x_{u_0}$, and $m=n$ is the first claim. For the second, apply the first to $p$ and to $q$ and subtract, using $x_{w}=A_px_{u_0}=A_qx_{u_0}$.
\end{proof}

Proposition~\ref{prop:section-path-eq} converts the sheaf layer from a decorative addition into a diagnostic with a stated meaning. Low sheaf energy on a retrieved neighborhood is evidence that the representation layer's compositional constraints are being respected \emph{on that neighborhood}, which is a statement about $\Lcal_{\mathrm{alg}}$ localized to the retrieved region rather than averaged over training. The converse fails, and usefully so: $A_p=A_q$ as operators does not imply that any particular $0$-cochain is a section, so the two diagnostics are genuinely different measurements.

\begin{remark}[Scope of the certification]
The proposition is a statement about the section, not about the operators. It asserts $(A_p-A_q)x_{u_0}=0$, not $A_p=A_q$ as operators. A retrieved subgraph can therefore have zero sheaf energy while the declared path equation fails outside the span of the observed entity states. This is the appropriate strength for a retrieval-time diagnostic: the system observes only finitely many entity states, and the guarantee is local to them.
\end{remark}

In PAGR, we use the energy as a \emph{consistency signal}: a retrieved subgraph with high local disagreement can be down-ranked, flagged, or split before generation.

\begin{principle}[Consistency is not truth]
Low sheaf energy indicates that local representations agree under the chosen restrictions. It does not imply that the underlying factual assertions are correct.
\end{principle}

\subsection{Sheaf cohomology and section-existence}
\label{sec:sheaf-cohomology}

The coboundary $\delta:C^0(G;\F)\to C^1(G;\F)$, where
\[
C^1(G;\F)\;\coloneqq\;\bigoplus_{e\in E_G}\F(e),
\]
is the first differential of the cellular cochain complex. Because $G$ is a graph (a one-dimensional CW complex), the complex terminates at degree~$1$:
\[
0\;\longrightarrow\; C^0(G;\F)\;\xrightarrow{\;\delta\;}\; C^1(G;\F)\;\longrightarrow\;0,
\]
with cohomology groups
\begin{align*}
H^0(G;\F)&\;=\;\ker\delta\quad\text{(global sections; cf.\ Proposition~\ref{prop:zero-energy})},\\
H^1(G;\F)&\;=\;C^1(G;\F)\,/\,\operatorname{im}(\delta).
\end{align*}

The primary diagnostic built from these groups is $H^0$, not $H^1$.

\begin{definition}[Consistency dimension]
\label{def:consistency-dim}
The \emph{consistency dimension} of a retrieved subgraph $H\subseteq G$ under $\rho$ is
\[
\sigma_\rho(H)\;\coloneqq\;\dim_k H^0(H;\F^{\rho}|_H).
\]
\end{definition}

So $\sigma_\rho(H)>0$ means that a nonzero globally consistent entity-state assignment exists on $H$, and $\sigma_\rho(H)=0$ means that the zero cochain is the only one. In general $\sigma_\rho(H)$ is computed directly as $\dim\ker\delta$, a rank computation on the coboundary; no further structure is assumed, and in particular the relation operators $A_a$ need not be invertible.

When they are, the quantity acquires a classical interpretation.

\begin{proposition}[Holonomy form under invertible transport]
\label{prop:holonomy}
Suppose every relation operator realized on a connected $H$ is an isomorphism. Then $\F^{\rho}|_H$ is a local system: all restriction maps are invertible, since $\F^{\rho}_{v\unlhd e}=\id$ and $\F^{\rho}_{u\unlhd e}=A_a$. Fixing a base vertex $v_0$, transport along any closed walk $\gamma$ based at $v_0$ is a well-defined automorphism $\operatorname{Hol}_\gamma$ of $V_{\tau(v_0)}$, and restriction to the base stalk identifies global sections with holonomy-invariant base states:
\[
\sigma_\rho(H)\;=\;\dim\!\!\bigcap_{\gamma\in\pi_1(H,v_0)}\!\!\ker\bigl(\operatorname{Hol}_\gamma-I\bigr),
\]
where $\pi_1(H,v_0)$ is free of rank $c$, so the intersection may be taken over the $c$ fundamental cycles of a spanning tree.
\end{proposition}
\begin{proof}
A section is determined by its value at $v_0$, since invertibility lets one transport uniquely along every edge of a spanning tree in either direction; and a base state extends consistently to a section precisely when it is unchanged by transport around each independent cycle.
\end{proof}

\begin{remark}[Why invertibility is needed for the holonomy form, but not for $\sigma_\rho$]
\label{rem:holonomy-caveat}
Without invertibility the holonomy statement is unavailable, and not merely unproved. Transport along a spanning tree may require traversing an edge against its direction, which asks for a preimage under $A_a$ that need neither exist nor be unique; and the round-trip operator of a cycle is undefined when an operator along it is singular. Take $H$ with sorts $u,v,w$, two relations $u\to w$ realized by $A_a$ and $A_c$, and one relation $v\to w$ realized by $A_b$: every spanning tree of $\bar H$ reaches $v$ through $w$, so tree propagation cannot even begin unless $A_b$ is invertible, and the holonomy of the $u$--$w$ cycle would be $A_c^{-1}A_a$. The consistency dimension is nevertheless perfectly well defined in this case and is computed from $\delta$ as usual. This is why Definition~\ref{def:consistency-dim} is stated as $\dim H^0$ rather than through holonomy: the rank computation is the general object, and Proposition~\ref{prop:holonomy} is the special case that explains what it measures.
\end{remark}

\begin{remark}[Why $H^1$ is not the consistency diagnostic]
\label{rem:h1-not-obstruction}
On a graph with $\sum_{v}d_{\tau(v)}=\sum_{e}d_{t(\lambda(e))}$---the balanced case, typical when all sorts have the same embedding dimension---the Euler characteristic identity $\dim H^0-\dim H^1=0$ forces $\sigma_\rho(H)=\dim H^1(H;\F^{\rho}|_H)$. In a triangle with consistent operators ($A_c=A_bA_a$) and stalks $k^3$, for instance, $\sigma_\rho=3$ and $\dim H^1=3$; with generic inconsistent operators, $\sigma_\rho=0$ and $\dim H^1=0$. So $H^1>0$ co-occurs with \emph{existence} of consistent configurations, not their absence. More generally, when edge stalks outweigh vertex stalks ($\chi<0$), the identity forces $\dim H^1\ge\sigma_\rho+|\chi|>0$ regardless of the operators, making $H^1>0$ vacuous as a diagnostic. The role of $H^1$ in this framework is topological: it measures the excess of overlap sections over those explainable by restriction from local sections, which is a property of the graph's cycle structure and the relative operator images rather than of operator consistency directly.
\end{remark}

The next theorem covers the multi-seed setting: Stage~II expands each seed to a local acyclic neighborhood, and the question is whether neighborhoods from distinct seeds can be assembled consistently. Its hypothesis is that the individual neighborhoods have vanishing $H^1$, and that condition requires care: it does \emph{not} follow from acyclicity alone.

\begin{lemma}[Vanishing on expansion forests]
\label{lem:forest-vanishing}
Let $H\subseteq G$ be a forest in which no vertex is the head of two edges---for instance an expansion tree produced by Stage~II from a single seed, in which each reached entity is recorded together with one witnessing traversal. Then $H^1(H;\F^{\rho}|_H)=0$.
\end{lemma}
\begin{proof}
By Definition~\ref{def:induced-sheaf} the restriction at the head of an edge $e=(u,v)$ labelled $a$ is $\F^{\rho}_{v\unlhd e}=\id_{V_{t(a)}}$, so $(\delta x)_e=x_v-A_ax_u$. Given $c\in C^1(H;\F^{\rho})$, root each component of $H$ and traverse outward. Each non-root vertex $v$ is the head of exactly one edge $e=(u,v)$, and $u$ has already been assigned, so setting $x_v\coloneqq c_e+A_ax_u$ assigns every vertex exactly once and gives $(\delta x)_e=c_e$ on every edge. Hence $\delta$ is surjective and $H^1=0$.
\end{proof}

\begin{remark}[Acyclicity alone is not sufficient]
\label{rem:forest-caveat}
The in-degree hypothesis cannot be dropped. If two edges of $H$ share a head, say $u_1\xrightarrow{a}w\xleftarrow{b}u_2$, then realizing a target $(c_{e_1},c_{e_2})$ requires
\[
x_w\in\bigl(c_{e_1}+\im A_a\bigr)\cap\bigl(c_{e_2}+\im A_b\bigr),
\]
which is empty for suitable $c$ unless $A_a$ and $A_b$ are surjective. Taking $V_i=k^2$ for all three sorts and $A_a,A_b$ of rank one gives $\dim H^1=1$ on this three-vertex tree. For a general cellular sheaf the situation is worse still, since neither restriction map need be surjective. An implementation that retains several witnessing traversals per reached entity should therefore verify $H^1(U;\F^{\rho}|_U)=H^1(V;\F^{\rho}|_V)=0$ directly---a rank computation on $\delta$---rather than infer it from acyclicity. Surjectivity of every $A_a$ is a sufficient alternative hypothesis.
\end{remark}

\begin{theorem}[Mayer--Vietoris gluing]
\label{thm:mv-patching}
Let $H=U\cup V$ where $U$ and $V$ are retrieved neighborhoods satisfying $H^1(U;\F^{\rho}|_U)=H^1(V;\F^{\rho}|_V)=0$---for instance expansion forests in the sense of Lemma~\ref{lem:forest-vanishing}. There is a short exact sequence of cochain complexes
\[
0\to C^{\bullet}(H;\F^{\rho})\to C^{\bullet}(U;\F^{\rho})\oplus C^{\bullet}(V;\F^{\rho})\to C^{\bullet}(U\cap V;\F^{\rho})\to 0,
\]
whose associated long exact sequence in cohomology, truncated using $H^1(U)=H^1(V)=0$, gives
\[
0 \;\to\; H^0(H;\F^{\rho}) \;\to\; H^0(U;\F^{\rho})\oplus H^0(V;\F^{\rho}) \;\xrightarrow{\;r\;}\; H^0(U\cap V;\F^{\rho}) \;\xrightarrow{\;\partial\;}\; H^1(H;\F^{\rho}) \;\to\; 0,
\]
where $r(s_U,s_V)=s_U|_{U\cap V}-s_V|_{U\cap V}$. Consequently:
\begin{enumerate}[leftmargin=*,label=(\roman*)]
\item A pair of local sections $(s_U,s_V)\in H^0(U;\F^{\rho})\oplus H^0(V;\F^{\rho})$ extends to a global section on $H$ if and only if $s_U|_{U\cap V}=s_V|_{U\cap V}$.
\item $\sigma_\rho(H)=\dim\ker(r)$: the consistency dimension of $H$ equals the number of independent ways local sections on $U$ and $V$ can be consistently matched on their overlap.
\item $H^1(H;\F^{\rho})\cong H^0(U\cap V;\F^{\rho})/\operatorname{im}(r)$: a topological residual measuring how much the overlap's section space exceeds what restriction from $U$ and $V$ can reach.
\end{enumerate}
\end{theorem}
\begin{proof}
The short exact sequence of cochain complexes is standard: in degree $0$, $C^0(H;\F^{\rho})\hookrightarrow C^0(U;\F^{\rho})\oplus C^0(V;\F^{\rho})\twoheadrightarrow C^0(U\cap V;\F^{\rho})$ is the restriction sequence, exact because every function on $U\cap V$ extends to $U$ and $V$ separately (the sheaf assigns spaces to vertices and edges, so surjectivity holds by definition). In degree $1$, both sides are edge-stalk direct sums and the same argument applies. The long exact sequence is then the standard algebraic consequence \citep{curry2014}. Truncation at $H^1(U)=H^1(V)=0$ yields the displayed five-term sequence. Claim~(i) is exactness at $H^0(U;\F^{\rho})\oplus H^0(V;\F^{\rho})$ and the identification of the kernel of $r$ with $H^0(H;\F^{\rho})$. Claims~(ii) and~(iii) follow from the five-term sequence by definition.
\end{proof}

The Euler characteristic identity
\[
\sigma_\rho(H) - \dim H^1(H;\F^{\rho}) \;=\; \sum_{v\in H}d_{\tau(v)} - \sum_{e\in H}d_{t(\lambda(e))}\;\eqqcolon\;\chi(H;\F^{\rho})
\]
is computable from the schema dimension vector alone, independently of $\rho$. When $\chi\ge0$ it implies $\sigma_\rho(H)\ge\dim H^1(H;\F^{\rho})\ge0$; when $\chi<0$ it forces $\dim H^1\ge|\chi|>0$ regardless of operators, which is why $H^1$ is not a useful consistency signal in that regime.

\begin{remark}[Two-tier consistency diagnosis]
\label{rem:h1-vs-energy}
The two diagnostics are complementary. $\Energy_{\F^{\rho}}(x)=0$ witnesses that a \emph{specific} observed configuration $x$ is globally consistent. $\sigma_\rho(H)>0$ certifies that a globally consistent configuration \emph{exists in principle}, independently of which particular entity states are observed: it is a property of the operators $\{A_a\}$ and the graph topology together. A retrieved context with $\sigma_\rho(H)>0$ but positive observed energy indicates that the entity states are not yet in a consistent configuration, but one exists and could be recovered by adjusting them (a training or inference task). A context with $\sigma_\rho(H)=0$ certifies that no \emph{nonzero} globally consistent assignment exists for the current operators---the zero cochain is always a section---which is a structural operator-inconsistency rather than a configurational gap. Reporting both $\Energy_{\F^{\rho}}$ and $\sigma_\rho$ is therefore more informative than energy alone.
\end{remark}

\section{A Unified Retrieval Operator}

We now combine the layers without collapsing them.

Let $q$ be a query.

\subsection{Stage I: geometric seeding}
For each eligible entity $x$, compute a semantic score $s_{\mathrm{geo}}(q,x)$. Select
\[
S_k(q)=\operatorname{TopK}_{x\in P}s_{\mathrm{geo}}(q,x).
\]
This stage maximizes semantic recall and may use dense text encoders, type-aware embeddings, or mixed geometry.

\subsection{Stage II: admissible symbolic expansion}
Let $L\subseteq Q_1^{\ast}$ be a regular language of allowed relation-label paths and let $h$ be a hop bound. For a path $p=a_n\cdots a_1$ in the schema write
\[
\Cl_{\T}(E)_p \;=\; \Cl_{\T}(E)_{a_n}\circ\cdots\circ\Cl_{\T}(E)_{a_1}
\;\subseteq\; P_{s(p)}\times P_{t(p)}
\]
for the relational composite of the certified closure along $p$, where $\Cl_{\T}(E)_a$ denotes the $a$-component of the closure and $\circ$ is relational composition. Define
\[
\Reach_{L,h}(S)=
\{y:\exists x\in S,\ \exists p\in L,\ |p|\le h,\ (x,y)\in \Cl_{\T}(E)_p\}.
\]
This is a bounded regular-path query over certified closure. Candidate/hypothesis edges may be explored in a separate channel but are never silently merged with this set.

\begin{remark}[Paths navigate; edges assert]
\label{rem:paths-vs-edges}
A witness path $p$ from $x$ to $y$ is a route, not an assertion: the framework does \emph{not} claim that the composite relation $p$ holds of $(x,y)$ unless $\T$ separately licenses that composite as an atom. Accordingly the certified retrieval output defined in Section~\ref{sec:noninterference} consists of the individual certified \emph{edges} traversed, each of which is an element of $\Cl_{\T}(E)$, together with the path structure presented as navigational metadata. Conflating the two would silently reintroduce the equation--inclusion confusion that Design Principle~\ref{prin:eq-incl} rules out.
\end{remark}

\subsection{Stage III: proof attachment}
Every certified edge exposed by Stage~II receives a certificate $\delta$ and provenance object $\Prov(\cdot)$. Results without a valid certificate remain candidates rather than certified context.

\subsection{Stage IV: geometric-algebraic scoring}
For a retrieved endpoint $y$, define components such as
\[
\begin{aligned}
s_{\mathrm{sem}}(y)&=s_{\mathrm{geo}}(q,y),\\
s_{\mathrm{path}}(y)&=g(|p_y|,\operatorname{type}(p_y)),\\
s_{\mathrm{rep}}(y)&=\exp(-\beta r_{p_y}(x,y)),\\
s_{\mathrm{prov}}(y)&=\varphi(\Prov(y)),\\
s_{\mathrm{cons}}(y)&=\exp(-\gamma \Energy_{\F}(x|_{N_y})).
\end{aligned}
\]
A ranking score may combine normalized components,
\[
\Score(q,y)=\sum_{j}\omega_j \widetilde{s}_j(q,y),
\qquad \omega_j\ge0,
\]
subject to a hard epistemic gate: ranking cannot convert a hypothesis into a certified fact.

\subsection{Stage V: graph diffusion}
Within the certified candidate subgraph, one may use personalized PageRank \citep{jeh2003} or another diffusion score, as GraphRAG systems built on entity graphs already do \citep{gutierrez2024}. Let $W$ be the column-stochastic transition matrix of the certified candidate subgraph and $v_q$ a probability distribution supported on seed nodes. (We use $W$ rather than the more common $P$ to avoid collision with the typed entity universe $P=\bigsqcup_i P_i$.) The personalization equation $\pi_q=\alpha W\pi_q+(1-\alpha)v_q$ has the unique solution
\[
\pi_q=(1-\alpha)(I-\alpha W)^{-1}v_q,
\qquad 0<\alpha<1,
\]
where $I-\alpha W$ is invertible because $\alpha\|W\|_1=\alpha<1$ bounds the spectral radius of $\alpha W$. The vector $v_q$, rather than a nonexistent ``query node,'' is the mathematically natural personalization object; the seed distribution may be taken proportional to the calibrated scores of Definition~\ref{def:calibrated} restricted to $S_k(q)$.

\subsection{Stage VI: generative synthesis}
The LLM receives a context package
\[
C_q=\{(f,\text{source}(f),\delta_f,\Prov(f),\mathsf{status}(f))\}
\]
plus optional candidate hypotheses explicitly marked as such. The model may summarize, compare, explain, or abstain, but it is not the authority that determines membership in $G^{\star}$.

\begin{algorithm}[t]
\caption{Proof-carrying geometric retrieval}
\label{alg:pcgr}
\begin{algorithmic}[1]
\Require query $q$, PAGR system $\KG$, seed count $k$, hop bound $h$
\State $S\gets \operatorname{TopK}_{x\in P}s_{\mathrm{geo}}(q,x)$
\State $C\gets \emptyset$
\ForAll{$x\in S$}
  \ForAll{admissible paths $p\in L$ with $|p|\le h$}
    \ForAll{edges $e$ traversed by a realization of $p$ from $x$ in $\Cl_{\T}(E)$}
      \State construct or retrieve derivation certificate $\delta_{e}$
      \If{$\delta_{e}$ verifies}
        \State compute provenance, representation residual, and consistency features for $e$
        \State add the certified edge $e$ to $C$, tagged with its witnessing path
      \EndIf
    \EndFor
  \EndFor
\EndFor
\State rank $C$ using gauge-invariant features and optional graph diffusion
\State \Return proof-carrying context $C_q$
\end{algorithmic}
\end{algorithm}

\section{Non-Interference, Soundness, and Conditional Completeness}
\label{sec:noninterference}

Collect the learned components into a single configuration
\[
\theta=(\Phi,\rho,\F,\omega,\text{LLM}),
\]
and let $\mathcal{R}^{\theta}_{\mathrm{cert}}(q)$ denote the set of atomic facts exposed to the generator as certified by Algorithm~\ref{alg:pcgr} under configuration $\theta$ (edges, in the sense of Remark~\ref{rem:paths-vs-edges}). Write $\Theta$ for the space of all such configurations.

The result below is often stated as a soundness theorem, but soundness is the weaker half of what the architecture actually buys. The stronger and more informative statement is an invariance property: the \emph{certification criterion is a function of the symbolic data alone}, uniformly over $\Theta$. This is a non-interference statement in the sense of \citet{goguen1982}, with epistemic status in place of a security level: no variation in the low-authority (learned) components can produce a variation in the high-authority (certified) predicate.

\begin{theorem}[Epistemic non-interference]
\label{thm:noninterference}
For every query $q$,
\[
\bigcup_{\theta\in\Theta}\mathcal{R}^{\theta}_{\mathrm{cert}}(q)\;\subseteq\;\Cl_{\T}(E).
\]
Equivalently, the \emph{criterion} by which a retrieved fact qualifies as certified is invariant under arbitrary replacement of $\Phi$, $\rho$, $\F$, the ranking weights $\omega$, and the language model. Replacement may well change \emph{which} certified facts are retrieved---the sets $\mathcal{R}^{\theta}_{\mathrm{cert}}(q)$ genuinely vary with $\theta$, and Theorem~\ref{thm:completeness} characterizes how---but not the standard a fact must meet to enter them: for every $\theta\in\Theta$ and every $f$,
\[
f\in\mathcal{R}^{\theta}_{\mathrm{cert}}(q)\ \Longrightarrow\ f\in\Cl_{\T}(E),
\]
and the predicate on the right-hand side does not depend on $\theta$. Learned components may therefore alter \emph{whether} a certified fact is retrieved, but not the criterion by which an exposed fact qualifies as certified.
\end{theorem}
\begin{proof}
Fix $\theta$. Learned components enter Algorithm~\ref{alg:pcgr} at exactly three points: the seed selection $S\gets\operatorname{TopK}s_{\mathrm{geo}}$, the computation of provenance, residual and consistency features, and the final ranking of $C$. The first only restricts which candidates are enumerated; the second and third only attach or reorder scores. None of the three adds an element to $C$, because the sole insertion into $C$ is guarded by the verification test on $\delta_{e}$. A certificate verifies only if its leaves resolve to admitted evidence and its internal nodes instantiate rules of $\T$, which by Definition~\ref{def:gstar} and Theorem~\ref{thm:closure} witnesses membership in $\Cl_{\T}(E)$. Since $\Cl_{\T}(E)$ depends only on $(E,\T)$ and not on $\theta$, the bound holds uniformly in $\theta$, and taking the union over $\Theta$ preserves it.
\end{proof}

\begin{corollary}[Retrieval soundness]
\label{cor:soundness}
Under Source Correctness and Rule Soundness,
\[
\mathcal{R}^{\theta}_{\mathrm{cert}}(q)\subseteq\At(\mathfrak{M})
\]
for every query $q$ and every configuration $\theta$.
\end{corollary}
\begin{proof}
Combine Theorem~\ref{thm:noninterference} with Theorem~\ref{thm:closure-sound}.
\end{proof}

\begin{remark}[What this theorem does and does not say]
\label{rem:not-vacuous}
Corollary~\ref{cor:soundness} is immediate from the construction, and we do not present it as a deep result: a system that emits only certificate-checked facts trivially emits only certified facts. The substantive content is in the universal quantifier of Theorem~\ref{thm:noninterference}. It is a statement about an entire design space $\Theta$, not about a single fixed system: an engineer may swap encoders, retrain relation operators, redesign the sheaf, retune $\omega$, or replace the generator, and the certified channel's correctness guarantee requires no re-examination. That is a maintenance and auditing guarantee, and it is precisely the property that fails in architectures where a similarity threshold or an LLM judgment can promote a hypothesis into certified context. The theorem's value is therefore as a specification that a system either satisfies or does not, and Appendix~\ref{app:invariants} lists the invariants that an implementation must enforce to satisfy it.
\end{remark}

\subsection{Conditional completeness}

Soundness bounds precision and says nothing about recall, yet recall is the entire justification for including a geometric layer. We therefore state the complementary bound, which is what makes the seed count $k$ and hop bound $h$ tunable against a stated criterion rather than by taste.

\begin{theorem}[Conditional completeness of certified retrieval]
\label{thm:completeness}
Fix a query $q$, seed count $k$, hop bound $h$, and admissible language $L$, and let $\theta\in\Theta$ be a configuration whose certificate checker is complete for $\T$. Write $S_k^{\theta}(q)$ for the seed set that $\theta$ produces, which depends on $\theta$ through $\Phi$. Let $f=(x',a,y')\in\Cl_{\T}(E)$ be a certified fact. If there exist a seed $x\in S_k^{\theta}(q)$ and a path $p\in L$ with $|p|\le h$ whose realization from $x$ traverses the edge $f$, then $f\in\mathcal{R}^{\theta}_{\mathrm{cert}}(q)$. Consequently
\[
\Cl_{\T}(E)\ \setminus\ \mathcal{R}^{\theta}_{\mathrm{cert}}(q)
\;=\;
\{f\in\Cl_{\T}(E): f\text{ lies on no }L\text{-admissible path of length}\le h\text{ from }S_k^{\theta}(q)\}.
\]
\end{theorem}
\begin{proof}
Since $x\in S_k(q)$ and $p\in L$ with $|p|\le h$, the nested enumeration of Algorithm~\ref{alg:pcgr} reaches the edge $f$ while expanding $p$ from $x$. A checker complete for $\T$ accepts the derivation of $f\in\Cl_{\T}(E)$ guaranteed by Theorem~\ref{thm:closure}, so the verification test succeeds and $f$ is added to $C$. The displayed identity is the contrapositive, together with the observation that by Theorem~\ref{thm:noninterference} nothing outside $\Cl_{\T}(E)$ is ever added.
\end{proof}

\begin{corollary}[Failure modes are attributable]
\label{cor:attributable}
A certified fact absent from $\mathcal{R}^{\theta}_{\mathrm{cert}}(q)$ fails for one or more of three attributable reasons: (i) \emph{seed miss}, no endpoint of a route to $f$ lies in $S_k^{\theta}(q)$; (ii) \emph{horizon miss}, every route exceeds $h$; (iii) \emph{grammar miss}, every route leaves $L$.
\end{corollary}

Corollary~\ref{cor:attributable} is the practical payoff. Each failure mode has a distinct remedy---improve $\Phi$ or raise $k$; raise $h$; widen $L$---and each is separately measurable by re-running retrieval with that single parameter relaxed. Theorem~\ref{thm:index-exact} shows that admissible-word feasibility can be decided on the bisimulation index alone, which rules out horizon and grammar explanations at the level of the path language; attributing a particular fact's absence still requires the certified graph (Remark~\ref{rem:index-caveat}). Reporting the three-way decomposition of recall loss is more informative than a single Recall@$k$ number, and we recommend it as the primary diagnostic for the semantic-recall axis of Section~\ref{sec:evaluation}.

Together, Theorems~\ref{thm:noninterference} and \ref{thm:completeness} express the intended division of labor precisely: the learned configuration $\theta$ controls the second containment and has no influence whatsoever on the first.

\section{Minimal Structures for Certified Retrieval}
\label{sec:minimal}

Two natural minimality questions arise for a PAGR system, and both bear directly on the retrieval operator rather than on the framework's presentation. Given the certified graph, what is the smallest structure that answers the same Stage~II expansions? And given the base set $E$, which facts are load-bearing? We answer both, and in each case the useful result is accompanied by a caveat that a naive statement would get wrong.

\subsection{Bounded bisimulation as a minimal Stage-II index}
\label{sec:bisim}

Let $G$ denote the certified graph (distinct from the base-change group $G_{\bm d}$ of Section~\ref{sec:moduli}): vertex set $P$, and for each $a\in Q_1$ an edge relation $\mathcal{E}_a=\Cl_{\T}(E)_a$. Write $x\xrightarrow{a}y$ for $(x,y)\in\mathcal{E}_a$ and $x\xrightarrow{w}y$ for a path spelling $w\in Q_1^{\ast}$.

\begin{definition}[Bounded forward bisimulation]
\label{def:hbisim}
Let $x\approx_0 y$ iff $x$ and $y$ have the same sort. Inductively, $x\approx_{m+1}y$ iff $x\approx_m y$ and for every $a\in Q_1$,
\[
\bigl\{[z]_{\approx_m} : x\xrightarrow{a}z\bigr\}
=\bigl\{[z]_{\approx_m} : y\xrightarrow{a}z\bigr\}.
\]
We call $\approx_h$ the $h$-bisimulation and write $G_h=G/\!\approx_h$ for the quotient, in which $[x]\xrightarrow{a}[y]$ iff $x'\xrightarrow{a}y'$ for some $x'\in[x]$, $y'\in[y]$.
\end{definition}

This is the classical notion underlying structural indexes for graph-structured data \citep{kaushik2002}, and $\approx_h$ is computable by partition refinement in $O(h\,|\mathcal{E}|)$, or the full bisimulation in $O(|\mathcal{E}|\log|P|)$ \citep{paige1987}.

\begin{lemma}[Word preservation]
\label{lem:word}
If $x\approx_h y$ then for every $w\in Q_1^{\ast}$ with $|w|\le h$, $x$ has an outgoing $w$-path iff $y$ does.
\end{lemma}
\begin{proof}
Induction on $|w|$; the case $|w|=0$ is immediate. Let $w=aw'$ with $|w|\le h$ and suppose $x\xrightarrow{a}z$ with $z$ having an outgoing $w'$-path. Since $x\approx_h y$ and $h\ge1$, the class $[z]_{\approx_{h-1}}$ occurs among the $a$-successor classes of $y$, so there is $z'$ with $y\xrightarrow{a}z'$ and $z'\approx_{h-1}z$. As $|w'|\le h-1$, the induction hypothesis gives $z'$ an outgoing $w'$-path.
\end{proof}

\begin{theorem}[Exact word realizability on the index]
\label{thm:index-exact}
For every $x\in P$ and every $w$ with $|w|\le h$: the class $[x]$ has an outgoing $w$-path in $G_h$ if and only if $x$ has an outgoing $w$-path in $G$. Consequently, for any regular $L$ and any $h'\le h$, membership of a seed in the set
\[
\mathrm{Dom}_{L,h'}=\{x\in P:\exists\, p\in L,\ |p|\le h',\ x\text{ has an outgoing }p\text{-path}\}
\]
is decided exactly on $G_h$.
\end{theorem}
\begin{proof}
($\Leftarrow$) The quotient map $\pi: G\to G_h$ is a graph homomorphism, so any $w$-path from $x$ in $G$ projects to a $w$-path from $[x]$ in $G_h$.

($\Rightarrow$) Let $[x]=C_0\xrightarrow{a_1}C_1\xrightarrow{a_2}\cdots\xrightarrow{a_n}C_n$ be a $w$-path in $G_h$ with $n=|w|\le h$. For each $k=1,\dots,n$, the definition of quotient edges supplies $\tilde{y}_{k-1}\in C_{k-1}$ and $y_k\in C_k$ with $\tilde{y}_{k-1}\xrightarrow{a_k}y_k$ in $G$. Set $y_0=\tilde{y}_0\in C_0$.

We claim $y_0$ has an outgoing $w$-path in $G$, which we construct one step at a time. We have $y_0\xrightarrow{a_1}y_1$. Now $y_1\in C_1$ and $\tilde{y}_1\in C_1$, so $y_1\approx_h\tilde{y}_1$. Since $\tilde{y}_1\xrightarrow{a_2}y_2$ and $|a_2\cdots a_n|=n-1\le h-1<h$, Lemma~\ref{lem:word} gives $y_1$ an outgoing $(a_2\cdots a_n)$-path in $G$. Concatenating with $y_0\xrightarrow{a_1}y_1$ yields a $w$-path from $y_0$.

Since $y_0\in C_0=[x]_{\approx_h}$ we have $x\approx_h y_0$, and $|w|=n\le h$, so Lemma~\ref{lem:word} transfers the $w$-path to $x$.
\end{proof}

The corresponding statement about \emph{endpoints} is false, and the failure is instructive.

\begin{remark}[The index does not preserve reached endpoints]
\label{rem:index-caveat}
It is tempting to conclude that $\Reach_{L,h}(S)$ can be computed on $G_h$ and read back. It cannot. Unfolding $[x]\xrightarrow{a}C_1\xrightarrow{a'}C_2$ yields a successor $z$ of $x$ with $z\approx_{h-1}y_1$ for some $y_1\in C_1$, but $\approx_{h-1}$ does not imply $\approx_h$, so $z$ need not lie in $C_1$ and the continuation need not land in $C_2$. Each unfolding step loses one level of the bisimulation, so only word realizability---which Lemma~\ref{lem:word} shows is preserved---survives to depth $h$. What remains true is one-sided: since the quotient map is a homomorphism,
\[
\Reach_{L,h}(S)\ \subseteq\ \pi^{-1}\!\left(\Reach^{G_h}_{L,h}(\pi(S))\right),
\]
so index expansion is a \emph{sound over-approximation}. It never discards a certified fact, and a validation pass against $G$ removes the false positives. This is the standard index-then-validate discipline \citep{kaushik2002}, and it is the only correct way to use $G_h$ in Stage~II.
\end{remark}

\begin{proposition}[Minimality of the index]
\label{prop:index-min}
$\approx_h$ is the coarsest equivalence refining the sort partition that is stable under $h$ rounds of the refinement of Definition~\ref{def:hbisim}. Hence $G_h$ has the fewest vertices among quotients through which the depth-$h$ unfolding of $G$ factors.
\end{proposition}
\begin{proof}
Immediate from the construction: $\approx_h$ is obtained by $h$ rounds of coarsest-stable partition refinement started from the sort partition, and each round produces the coarsest refinement satisfying its stability condition.
\end{proof}

We state minimality relative to bisimulation rather than to word equivalence deliberately. Trace equivalence---identifying nodes with the same realizable word sets---is strictly coarser on nondeterministic graphs, and knowledge graphs are nondeterministic in the relevant sense, since a node may have many $a$-successors. But deciding trace equivalence is PSPACE-complete, whereas $\approx_h$ is near-linear. The bisimulation quotient is the largest compression that remains cheap, which is why it is the one used in practice.

The concrete payoff is narrower than it may first appear, and Remark~\ref{rem:index-caveat} bounds it. Theorem~\ref{thm:index-exact} decides admissible-word \emph{feasibility} exactly, so it can rule out candidate horizon and grammar explanations at the path-language level---if no admissible word of length $\le h$ leaves a seed, nothing is reachable from that seed and it can be discarded without expanding $G$. Attributing the absence of a \emph{particular} certified fact $f$ in the sense of Corollary~\ref{cor:attributable} is a different question, since it asks whether some realization of an admissible word traverses $f$, and that requires validation against $G$: the index preserves neither reached endpoints nor individual traversed edges.

\subsection{Minimal bases and the value of redundancy}
\label{sec:basemin}

The second question concerns $E$ itself. Some admitted base facts may be derivable from the others.

\begin{definition}[Redundant base fact]
$f\in E$ is \emph{redundant} if $f\in\Cl_{\T}(E\setminus\{f\})$.
\end{definition}

\begin{proposition}[Provenance decides redundancy]
\label{prop:prov-redundant}
Let $\Prov(f)\in\mathbb{N}[X]$ be the provenance of $f$ computed over $E$. Then $f$ is redundant if and only if the Boolean evaluation of $\Prov(f)$ under $x_f\mapsto 0$ and $x_e\mapsto 1$ for all $e\neq f$ is $1$.
\end{proposition}
\begin{proof}
The monomials of $\Prov(f)$ enumerate the admitted derivations of $f$, one of which is the trivial derivation contributing $x_f$. Setting $x_f\mapsto0$ annihilates exactly the monomials whose derivations use $f$ as a leaf. A surviving monomial is therefore a derivation of $f$ from $E\setminus\{f\}$, and conversely.
\end{proof}

So the redundancy test requires no additional machinery: it is a semiring evaluation of an object the system already maintains for certification.

\begin{definition}[Base core]
$E'\subseteq E$ is a \emph{base core} if $\Cl_{\T}(E')=\Cl_{\T}(E)$ and no proper subset of $E'$ has this property.
\end{definition}

Greedily testing facts by Proposition~\ref{prop:prov-redundant} and deleting redundant ones yields a base core. Two limits should be recorded so that this is not oversold.

\begin{proposition}[Minimum bases are intractable]
\label{prop:min-np}
Computing a base core of minimum cardinality is NP-hard.
\end{proposition}
\begin{proof}
Take $\T$ to consist of the single transitivity rule $R;R\subseteq R$, so that $\Cl_{\T}(E)$ is the transitive closure of the digraph $E$. A minimum-cardinality $E'\subseteq E$ with the same transitive closure is exactly a minimum equivalent digraph of $E$, which is NP-hard \citep{garey1979}; see \citet{aho1972} for the relationship to transitive reduction.
\end{proof}

\begin{remark}[No canonical simplest knowledge base]
\label{rem:no-canonical}
Base cores are not unique, and unlike the core of a relational structure under homomorphism equivalence \citep{chandra1977} they need not be isomorphic to one another. More fundamentally, equivalence of Datalog programs is undecidable \citep{shmueli1993}, so there is no general procedure for deciding whether two PAGR systems carry the same inferential content. Minimality claims in this framework are therefore always relative to a fixed rule set and a fixed query class.
\end{remark}

Most importantly, minimizing $E$ is not epistemically free, and a system should usually decline to do it.

\begin{principle}[Redundancy in the base is trust redundancy]
\label{prin:redundancy}
Deleting a redundant $f$ from $E$ changes its status from $\mathsf{base}$ to $\mathsf{derived}$. Its certification then depends on rule soundness in addition to source correctness, lengthening its trust path, and by Section~\ref{sec:retraction} it becomes vulnerable to retraction of the facts that now derive it. Base minimization is therefore an analysis and compression tool, not a maintenance policy: a fact supported both by direct evidence and by derivation is more robust than one supported by either alone.
\end{principle}

This suggests reporting, rather than eliminating, redundancy. Define the \emph{evidential multiplicity} of a certified fact as the number of monomials in $\Prov(f)$ that survive deletion of any single base fact. Facts of multiplicity zero are single points of failure in the evidence graph; their distribution is a more informative robustness statistic than the size of $E$, and it is computable from provenance alone.

\section{Quiver Algebra Beyond Notation}

It is tempting either to overuse representation theory or to dismiss it as decorative. We identify four concrete roles for quiver algebra.

\subsection{Typed operator composition}
The first role is exact: $A_p$ is defined only for type-compatible paths. Ill-typed relation composition is impossible at the algebraic level.

\subsection{Constraint regularization}
\label{sec:constraint-reg}
Path identities yield testable operator constraints. For example, if $a$ is an inverse of $b$ in the schema, the equations
\[
ba=e_i,\qquad ab=e_j
\]
induce
\[
A_bA_a\approx I_{V_i},\qquad A_aA_b\approx I_{V_j}.
\]
If $c=ba$ is a declared compositional identity, enforce $A_c\approx A_bA_a$. Note that the inverse-pair equations are not admissible relations in the sense of Remark~\ref{rem:admissible}, since they equate a length-two path with a length-zero one; they are perfectly usable as operator constraints, but the representation-theoretic classification results should not be applied to the resulting quotient.

\subsection{Idempotent decomposition}
For an $\A_Q$-module $V$,
\[
V=\bigoplus_{i\in Q_0}e_iV.
\]
This is the algebraic origin of type-specific embedding spaces. It avoids forcing heterogeneous entities into one undifferentiated vector space.

\subsection{Representation type of schema quivers: a scope theorem}
\label{sec:reptype}

A framework that invokes quiver representations should say precisely which classical results are available to it. For PAGR the answer is sharply negative, and stating it is useful: it rules out an entire family of approaches that a reader might otherwise expect us to pursue, and it explains why the diagnostics of Section~\ref{sec:diagnostics} are restricted to invariants of a \emph{given} representation.

Recall the classification of representation type. Write $\bar Q$ for the underlying undirected multigraph of $Q$, obtained by forgetting arrow orientation.

\begin{theorem}[Gabriel; Kac; Drozd]
\label{thm:reptype}
Let $Q$ be a connected finite quiver and $k$ algebraically closed, and let $\Rep_k(Q)$ denote the category of finite-dimensional representations of $Q$. Say $Q$ has \emph{finite representation type} if $\Rep_k(Q)$ has finitely many indecomposables up to isomorphism, and \emph{tame} or \emph{wild} type in the usual sense for that category.
\begin{enumerate}[leftmargin=*,label=(\roman*)]
    \item \citep{gabriel1972} $Q$ has finite representation type if and only if $\bar Q$ is a simply-laced Dynkin diagram $A_n$, $D_n$, $E_6$, $E_7$, or $E_8$. In that case the indecomposables correspond bijectively to the positive roots of the associated root system.
    \item \citep{kac1980} $Q$ has tame type if and only if $\bar Q$ is an affine (Euclidean) Dynkin diagram $\tilde A_n$, $\tilde D_n$, $\tilde E_6$, $\tilde E_7$, or $\tilde E_8$. In all types, the dimension vectors of indecomposables are exactly the positive roots of the associated Kac--Moody root system.
    \item \citep{drozd1980} Otherwise $Q$ is wild: its representation classification problem contains that of every finite-dimensional $k$-algebra.
\end{enumerate}
\end{theorem}

\begin{remark}[Which category the classification concerns]
\label{rem:fd-scope}
The algebra $kQ$ is finite dimensional precisely when $Q$ has no oriented cycles, and PAGR deliberately permits schemas that do: inverse pairs, self-relations such as $\mathsf{refines}:M\to M$, and directed cycles among sorts are all natural in a knowledge schema. We therefore state Theorem~\ref{thm:reptype} for $\Rep_k(Q)$, the category of finite-dimensional representations---which is what the representation layer actually parameterizes---rather than for $kQ$ as a finite-dimensional algebra. This is the setting of Kac's theorem, which covers quivers with loops and oriented cycles, and in which the trichotomy depends only on $\bar Q$. Two consequences are worth noting. Orientation is immaterial, so it makes no difference whether a double edge arises from an inverse pair or from two relations in the same direction. And clause~(iii) is quoted in its standard form for finite-dimensional algebras; when $Q$ has an oriented cycle, $kQ$ is infinite dimensional and it is the classification problem for $\Rep_k(Q)$, not for the algebra $kQ$ in Drozd's sense, that is wild. Statements elsewhere in this paper that genuinely require $kQ$ to be finite dimensional are flagged where they occur; the hereditary property used in Section~\ref{sec:diagnostics} is not among them, since path algebras satisfy $\operatorname{gl.dim} kQ\le1$ for every $Q$.
\end{remark}

Two features of this classification matter for schema design. First, representation type depends only on $\bar Q$, so \emph{the orientation of relations is irrelevant}: declaring $\mathsf{authored}:A\to D$ rather than $\mathsf{authoredBy}:D\to A$ cannot change which regime a schema falls into. Second, every Dynkin diagram is a tree, and every affine Dynkin diagram is either a tree or a single cycle. This yields a purely arithmetic test.

\begin{corollary}[Cyclomatic criterion]
\label{cor:cyclomatic}
Let $Q$ be a connected schema quiver with $n=|Q_0|$ sorts and $m=|Q_1|$ primitive relation types, and let $c=m-n+1$ be the cyclomatic number of $\bar Q$. Then:
\begin{enumerate}[leftmargin=*,label=(\roman*)]
    \item finite type requires $c=0$, i.e.\ $m=n-1$, with $\bar Q$ a Dynkin tree;
    \item tame type requires $c\le 1$;
    \item $c\ge 2$---equivalently $m\ge n+1$---forces wild type.
\end{enumerate}
\end{corollary}
\begin{proof}
Dynkin diagrams are trees, so $c=0$; affine Dynkin diagrams are either trees ($\tilde D$, $\tilde E$) or a single cycle ($\tilde A_n$, including $\tilde A_1$ the double edge and $\tilde A_0$ the single loop), so $c\le1$. By Theorem~\ref{thm:reptype} anything else is wild.
\end{proof}

\begin{corollary}[Common schema features force wildness]
\label{cor:wild-features}
Call a \emph{cycle feature} of $Q$ any of the following configurations in $\bar Q$: a self-relation $a:i\to i$; a pair of relation types between the same two sorts, in either orientation---so in particular an inverse pair $a:i\to j$, $b:j\to i$; or an undirected cycle through three or more distinct sorts. Then a connected schema is of wild type as soon as it contains
\begin{enumerate}[leftmargin=*,label=(\roman*)]
\item two \emph{edge-disjoint} cycle features, or
\item three or more relation types between a single pair of sorts.
\end{enumerate}
A \emph{single} cycle feature is not enough: one self-relation alone gives $\tilde A_0$ and one pair of parallel relations alone gives $\tilde A_1$, the Kronecker quiver, both of which have $c=1$ and are tame.
\end{corollary}
\begin{proof}
Each cycle feature is a cycle of $\bar Q$ in the graph-theoretic sense---a loop, a $2$-cycle, or a longer cycle. Edge-disjoint cycles are linearly independent in the cycle space of $\bar Q$, whose dimension is $c$, so case~(i) gives $c\ge2$. In case~(ii), the two sorts and their $k\ge3$ connecting edges already give $c\ge k-1\ge2$. Corollary~\ref{cor:cyclomatic}(iii) then applies. The final claim is Corollary~\ref{cor:cyclomatic}(i)--(ii): a single loop has $n=m=1$ and a single parallel pair has $n=m=2$, so $c=1$ in both cases.
\end{proof}

\begin{remark}[Why the features must be edge-disjoint]
\label{rem:edge-disjoint}
The disjointness hypothesis cannot be dropped, and dropping it is a natural mistake, because two informally stated features can name the same edges. The quiver $i\xrightarrow{a}j$, $j\xrightarrow{b}i$ is at once ``an inverse pair'' and ``two distinct relation types between the same pair of sorts,'' yet these are one cycle feature and not two: representation type depends only on $\bar Q$, in which both descriptions collapse to a single double edge with $c=1$. The quiver is $\tilde A_1$ and tame.
\end{remark}

The practical reading is direct. A schema with a single self-relation and nothing else---the worked example's $\mathsf{refines}:M\to M$ in isolation---is tame, being $\tilde A_0$, whose indecomposables are Jordan blocks. A schema with one inverse pair and nothing else is $\tilde A_1$, the Kronecker quiver, also tame; the same holds for two distinct relation types between the same pair of sorts, since representation type depends only on $\bar Q$ and both cases yield a double edge. Wildness requires $c\ge2$, so by Corollary~\ref{cor:wild-features} it is combinations that force it: an inverse pair \emph{together with} a self-relation on a disjoint set of arrows, three or more relation types between one pair of sorts, or any two edge-disjoint cycle features. The criterion is mild: by Corollary~\ref{cor:cyclomatic} a connected schema need only carry one more primitive relation type than it has entity sorts. Many schemas will satisfy it, but we deliberately make no claim about how many. Establishing such a claim would require a stated counting convention for sorts and primitive relation types, applied to a sample of published schemas, and we have not carried out that study; nor does anything in PAGR require the relation vocabulary to outgrow the sort vocabulary, since a schema may keep a small set of primitive relations and carry its semantic variation in attributes, Horn rules, or evidence metadata instead. What the scope theorem supplies is a \emph{decision procedure} rather than a census: given $Q$, compute $c$ and read off which classical tools are legitimately available.

\begin{remark}[This is a delimitation, not a defect]
\label{rem:wild-scope}
Theorem~\ref{thm:reptype} does not weaken PAGR; it tells us which tools are unavailable and thereby which are worth building. Three consequences are worth recording. \emph{(a)} In the wild regime no tractable Dynkin- or affine-style classification of the indecomposables is to be expected, so the representation layer must be learned and validated empirically rather than selected from a classification---this is the formal reason the hypotheses of Section~\ref{sec:hypotheses} are stated as experiments. \emph{(b)} Auslander--Reiten theory and quiver-Grassmannian techniques remain perfectly well defined in the wild regime; what fails is the manageable global classification of indecomposables that makes them a practical design tool, so results obtained for Dynkin or affine schemas should not be transported to ours without checking that they survive. \emph{(c)} Wildness concerns the classification of \emph{all} representations and says nothing about a given one. Invariants of a particular learned $\rho$---its endomorphism algebra, its direct-sum decomposition, its kernels and images---remain computable, which is precisely the scope of Section~\ref{sec:diagnostics}.
\end{remark}

\begin{remark}[Scope of the theorem under path equations]
\label{rem:reptype-bound}
Theorem~\ref{thm:reptype} classifies $kQ$, i.e.\ the hereditary case $J=0$. Imposing path equations can only cut down the module category, so a quotient $\A_Q=kQ/J$ may have strictly smaller representation type than $kQ$. The representation type of a bound quiver algebra is \emph{not} determined by $\bar Q$ alone and its classification is substantially more delicate; moreover the standard theory presumes $J$ admissible, which Remark~\ref{rem:admissible} shows some natural PAGR constraints violate. We therefore state the scope theorem for the unconstrained representation layer, which is the regime in which relation operators are actually trained, and make no claim about $\A_Q$.
\end{remark}

\subsection{Representation-theoretic diagnostics}
\label{sec:diagnostics}
For small schemas, one may study kernels, images, invariant subspaces, or decompositions of composite maps to identify redundant relation channels or bottlenecks. By Remark~\ref{rem:wild-scope}(c) these remain available in wild type, since each is an invariant of a fixed $\rho$ rather than a classification of all of them. Homological tools such as $\operatorname{Ext}$ become relevant when there is an actual module-theoretic question---for instance, $\operatorname{Ext}^1(M,N)$ classifies the extensions of $M$ by $N$ and therefore measures the obstruction to a decomposition being a direct sum---not merely because a graph exists.

The consistency dimension $\sigma_\rho(H)$ of Section~\ref{sec:sheaf-cohomology} provides a global diagnostic computable on the retrieved subgraph without access to the full certified closure: it reports the dimension of the space of globally consistent entity-state assignments, and $\sigma_\rho(H)=0$ is a certificate that no reassignment of entity states can render the retrieved context consistent.

A further homological tool operates at the level of the algebra itself. The \emph{Hochschild cohomology} $HH^*(A_Q)$ of the schema algebra $A_Q=kQ/J$ governs the infinitesimal deformations of its representations: $HH^2(A_Q)$ classifies first-order algebra deformations, and for a specific representation $\rho$ the obstruction to deforming $\rho$ as an $A_Q$-module is controlled by $\Ext^2_{A_Q}(M_\rho,M_\rho)$, where $M_\rho$ is the module defined by $\rho$. In the hereditary case $J=0$ the global dimension of $kQ$ is $1$, so $\Ext^n_{kQ}(M,N)=0$ for all $n\ge2$ and all modules $M,N$. In particular $\Ext^2_{kQ}(M_\rho,M_\rho)=0$, so the deformation problem for any single representation is \emph{unobstructed}: every first-order deformation extends to all orders. (Formal rigidity, the stronger condition $\Ext^1(M_\rho,M_\rho)=0$, does not hold in general.) The corresponding statement about the algebra is a separate one and does not follow from the global dimension, since $HH^n(A)=\Ext^n_{A^e}(A,A)$ is computed over the enveloping algebra $A^e=A\otimes_k A^{\mathrm{op}}$ rather than over $A$. It nonetheless holds here: a path algebra admits a two-term projective bimodule resolution, so its Hochschild dimension is at most $1$ and $HH^n(kQ)=0$ for $n\ge2$ \citep{happel1989}, meaning $kQ$ has no nontrivial infinitesimal algebra deformations either. This is a theorem the framework inherits for free whenever path equations are absent, and it justifies treating unconstrained relation operators as continuously trainable without algebraic obstruction. When $J\neq0$ the vanishing may fail: $\Ext^2_{A_Q}(M_\rho,M_\rho)$ supplies a nonzero obstruction space, permitting \emph{obstructed parameter regimes} in which relation operators may resist simultaneous consistent training; whether an obstruction is actually met depends on whether a particular obstruction class is nonzero---a structural precursor to the compositional generalization gaps of Hypothesis H1 in Section~\ref{sec:hypotheses}.

\section{Learning Without Corrupting Ground Truth}

The architecture supports learning at several levels.

\subsection{Learned entity embeddings}
Train or freeze $\phi_i$ using text, images, tabular features, or multimodal encoders.

\subsection{Learned relation operators}
Given positive certified triples $(x,a,y)$, fit $A_a$ by minimizing a loss such as
\[
\Lcal_{\mathrm{rel}}=\sum_{(x,a,y)\in E}\ell(A_a z_x,z_y)+\Lcal_{\mathrm{neg}}.
\]
Add algebraic constraints $\Lcal_{\mathrm{alg}}$ and sheaf energy regularization.

\subsection{Joint objective}
A general objective is
\[
\Lcal=
\lambda_1\Lcal_{\mathrm{retr}}
+\lambda_2\Lcal_{\mathrm{rel}}
+\lambda_3\Lcal_{\mathrm{alg}}
+\lambda_4\Lcal_{\mathrm{sheaf}}
+\lambda_5\Lcal_{\mathrm{cal}}.
\]
Here $\Lcal_{\mathrm{cal}}$ may calibrate hypothesis scores against held-out certified edges.

\subsection{Hypothesis queue}
High-scoring predicted edges enter a separate queue
\[
H=\{(x,a,y,s): s\ge \tau\},
\]
where they can be used for discovery, active learning, or evidence acquisition. They remain outside $G^{\star}$ until admitted by source evidence or a symbolic proof.

This separation permits aggressive learning without contaminating certified context.

\section{Ground-Truth Establishment as an Evidence Protocol}
\label{sec:evidence}

A retrieval system cannot mathematically manufacture truth from text. It can, however, make the establishment of its \emph{internal} ground-truth set explicit.

\begin{definition}[Evidence object]
An evidence object for a base fact $f$ is
\[
\epsilon_f=(\mathrm{sourceID},\mathrm{locator},\mathrm{spanHash},\mathrm{extract},\mathrm{timestamp},\mathrm{validator}),
\]
where the locator identifies the relevant source region and the hash binds the admitted fact to a stable source representation.
\end{definition}

\begin{definition}[Admission predicate]
$\operatorname{Admit}(f,\epsilon_f)\in\{0,1\}$ is a deterministic policy deciding whether evidence is sufficient for $f$ to enter $E$.
\end{definition}

The policy may be domain-specific---peer-reviewed publication, signed database record, official API response, human curator approval, consensus threshold, or cryptographic attestation---but the mathematical framework is invariant to that choice.

\begin{principle}[Ground truth is protocol-relative]
\label{prin:protocol}
The certified base $E$ is not ``all truths in the world.'' It is the set of assertions admitted under an explicit evidence protocol. Reproducibility requires versioning both $E$ and the admission policy.
\end{principle}

\section{Contradiction and Paraconsistent Extensions}

Real knowledge bases may contain conflicting sources. Classical closure over an inconsistent theory can be problematic if arbitrary negation is allowed. The positive Horn fragment used above avoids explosion because it contains no unrestricted classical negation. For domains requiring explicit contradiction, one may maintain signed predicates $R^{+}$ and $R^{-}$ or adopt a paraconsistent truth lattice.

A simple four-valued status for an atom $f$ is Belnap's bilattice \citep{belnap1977}
\[
\{\mathbf{T},\mathbf{F},\mathbf{B},\mathbf{N}\},
\]
meaning supported true, supported false, both, or neither. Geometry may rank evidence on either side, while the generator must expose the conflict rather than collapse it to one scalar probability.

This is an extension rather than a required part of the core PAGR framework.

\section{Retraction and Non-Monotonic Maintenance}
\label{sec:retraction}

The closure of Section~\ref{sec:certified} is monotone, so it models a knowledge base that only grows. Real corpora do not behave this way: sources are corrected, superseded, withdrawn, or found to have been misextracted, and an admitted base fact may have to be removed. Because Design Principle~\ref{prin:no-promotion} forbids repairing this by re-scoring, the framework needs an explicit retraction mechanism, and provenance supplies one.

Suppose $e\in E$ is retracted. The facts that must be reconsidered are exactly those whose provenance polynomial mentions $x_e$; a fact $f$ survives if and only if $\Prov(f)$ has a monomial none of whose variables has been retracted, i.e.\ if $\Prov(f)$ evaluated with $x_e\mapsto0$ and all surviving generators mapped to $1$ is nonzero in the Boolean semiring. Provenance thus doubles as a truth-maintenance index: no re-derivation from scratch is needed, and the surviving derivations are exactly the certificates that remain valid.

Two formal routes are available for the general case. One may move from the polynomial semiring $\mathbb{N}[X]$ to an $m$-semiring equipped with a monus operator, which is what is required for a difference operation on annotated relations to be well behaved \citep{geerts2010}; or one may keep the positive semiring and treat retraction as recomputation of the least fixed point over the reduced base, using the provenance index only to bound which facts can possibly be affected. The second is simpler and adequate whenever retractions are rare relative to queries, which is the common regime.

\begin{principle}[Retraction is an evidence event]
\label{prin:retraction}
Removing a fact from $E$ is an action of the admission policy, not of the retrieval system. Its effect on $\Cl_{\T}(E)$ is determined by provenance, and the recomputation preserves both Theorem~\ref{thm:noninterference} and Corollary~\ref{cor:soundness}, since both are stated relative to whatever $E$ currently is.
\end{principle}

Versioning $E$ and the admission policy, as Design Principle~\ref{prin:protocol} already requires, is what makes a retraction auditable after the fact. Design Principle~\ref{prin:redundancy} is the counterpart on the other side: redundant base facts are precisely what makes a certified closure robust to retraction, which is why $E$ should not be minimized in production.

\section{Temporal and Dynamic Knowledge}

Let facts be indexed by validity intervals or event time. Replace $E_a$ by
\[
E_a\subseteq P_i\times P_j\times\mathbb{T}.
\]
Rules become temporal Horn rules, and provenance certificates inherit temporal constraints. A query at time $t$ retrieves from a temporal slice or interval algebra. The same epistemic separation continues to apply: learned temporal embeddings can predict changes, but predictions remain hypotheses.

\section{Complexity and Computational Boundaries}

\subsection{Symbolic closure}
For a fixed finite rule set $\T$ and finite $P$, naive saturation terminates but may be expensive. The relevant classical bounds are standard \citep{abiteboul1995}: evaluating a fixed positive Datalog program is \textsc{PTime}-complete in data complexity, and for a program whose rule bodies contain at most $v$ distinct variables the closure is computed in $O(|\T|\cdot|P|^{v})$ time by naive saturation. Semi-naive evaluation avoids re-deriving old facts and reduces the constant substantially; combined complexity is \textsc{ExpTime}-complete, which is why the rule set should be treated as fixed and small. In practice, materialization can be partial and query-driven.

\subsection{Bounded path retrieval}
Let $\mathcal{A}$ be a deterministic automaton for $L$. Evaluating the bounded regular-path query $\Reach_{L,h}$ amounts to a breadth-first search of depth $h$ in the product of the certified graph with $\mathcal{A}$, costing $O(h\cdot|\mathcal{A}|\cdot|\Cl_{\T}(E)|)$ edge traversals from a single seed and hence $O(k\cdot h\cdot|\mathcal{A}|\cdot|\Cl_{\T}(E)|)$ for $k$ seeds. Unbounded regular-path query evaluation is \textsc{NL}-complete in data complexity; the hop bound $h$ is what keeps the practical cost linear in the certified edge count and simultaneously controls semantic drift. Note the direct tension with Corollary~\ref{cor:attributable}: raising $h$ or widening $L$ removes horizon and grammar misses at exactly this linear cost. The bisimulation index of Section~\ref{sec:bisim} replaces $|\Cl_{\T}(E)|$ by $|G_h|$ in the pruning phase, at a one-off preprocessing cost of $O(h\,|\Cl_{\T}(E)|)$ and the price of a validation pass.

\subsection{Index construction and base analysis}
Computing $\approx_h$ is $O(h\,|\mathcal{E}|)$ by partition refinement, or $O(|\mathcal{E}|\log|P|)$ for the full bisimulation \citep{paige1987}. Redundancy testing over the base (Proposition~\ref{prop:prov-redundant}) is a semiring evaluation per fact, linear in the size of the stored provenance; computing a \emph{minimum} base is NP-hard (Proposition~\ref{prop:min-np}), so only the greedy variant is practical.

\subsection{Representation scoring}
A path operator $A_p$ requires $|p|$ matrix-vector applications. When relation operators are sparse, orthogonal, diagonal, block-structured, or low-rank, costs can be substantially reduced.

\subsection{Sheaf energy}
On a retrieved subgraph, computing $\delta x$ is linear in incidences times the local restriction-map cost. Full spectral decomposition of $L_{\F}$ is unnecessary for simple consistency scoring.

\section{Evaluation Framework}
\label{sec:evaluation}

A PAGR system should not be evaluated by answer accuracy alone. We propose seven independent axes: six measuring properties of the retrieval system itself, and one measuring the ingestion policy on which soundness depends.

\begin{longtable}{p{0.18\textwidth}p{0.31\textwidth}p{0.41\textwidth}}
\toprule
Axis & Question & Example metric \\
\midrule
\endhead
Semantic recall & Did geometry find the right region? & Recall@$k$ of evidence-bearing seed neighborhoods, decomposed into seed / horizon / grammar misses per Corollary~\ref{cor:attributable}. \\
Symbolic precision & Are admitted compositions sound? & Precision of derived relations under expert or held-out validation. \\
Proof coverage & How much retrieved certified context has machine-checkable witnesses? & Fraction of certified outputs with valid certificates; target $1$. \\
Representation coherence & Does learned geometry respect schema equations? & Mean normalized path-equation residual. \\
Local consistency & Do retrieved representations glue under the sheaf? & Distribution of $\Energy_{\F}$ over retrieved neighborhoods; consistency dimension $\sigma_\rho(H)=\dim H^0(H;\F^{\rho}|_H)$ (Section~\ref{sec:sheaf-cohomology}) as the dimension of the space of globally consistent configurations. \\
Generation faithfulness & Does the LLM state only what the context supports? & Claim-level entailment/attribution evaluation \citep{gao2023}, separate from citation presence. \\
Attestation fidelity & Do admitted base facts actually appear in their cited spans? & Attestation error rate of $\operatorname{Admit}$ on a held-out audited sample (Section~\ref{sec:extraction}). \\
\bottomrule
\end{longtable}

The first six axes are properties of the retrieval system and can be measured without inspecting the ingestion pipeline. The seventh, attestation fidelity, is a property of the admission policy and is the empirical content of Assumption Source Correctness: reporting it is what prevents Corollary~\ref{cor:soundness} from being read as a stronger guarantee than it actually provides.

Additional ablations should remove geometry, algebraic constraints, sheaf consistency, provenance, and graph diffusion separately. The key empirical hypothesis is that the layers are complementary: geometry should improve recall, while proof-carrying symbolic constraints should protect precision.

\section{Proposed Research Hypotheses}
\label{sec:hypotheses}

The framework yields experimentally testable hypotheses. Each is stated so that a negative result is informative; H1, H2, and H5 are cheap enough to run on a standard benchmark such as FB15k-237 or WN18RR without building a full retrieval stack, and we regard them as the minimum evidence the framework should be asked to produce.

\paragraph{H1: Algebraic regularization improves compositional retrieval.}
Relation operators trained with path-equation constraints should generalize better to unseen multi-hop compositions than unconstrained operators.

\paragraph{H2: Sheaf energy predicts retrieval incoherence.}
High $\Energy_{\F}$ neighborhoods should correlate with contradictory or semantically unstable retrieval contexts. Proposition~\ref{prop:section-path-eq} gives this hypothesis a mechanism to test rather than a correlation to hope for: energy should track localized path-equation violation, so the two quantities should be measured jointly and their correlation reported.

\paragraph{H3: Epistemic separation improves factual precision at fixed recall.}
Allowing learned hypotheses to influence ranking but not certification should preserve discovery benefits while reducing unsupported generated claims.

\paragraph{H4: Type-specific geometry dominates a single global metric in heterogeneous graphs.}
Separate spaces $\M_i$ and typed maps $A_a$ should outperform a single homogeneous embedding space when relation signatures differ substantially.

\paragraph{H5: Gauge-invariant scores improve reproducibility across equivalent parameterizations.}
Retrieval based on invariant residuals should be stable under isometric basis changes. Remark~\ref{rem:gauge-gap} sharpens this into a falsifiable pair of predictions: rankings should be exactly invariant under $\Iso(\bm d)$ and should \emph{measurably drift} under general $G_{\bm d}$ base changes unless operators are orthogonally parameterized. Observing drift under $G_{\bm d}$ is a confirmation, not a bug, and its magnitude quantifies how much a moduli-space account of the representation layer overstates the true invariance.

\section{Worked Abstract Example}

Consider three entity types
\[
\mathsf{Document},\quad \mathsf{Claim},\quad \mathsf{Method}
\]
and primitive arrows
\[
\mathsf{states}:D\to C,\quad
\mathsf{supports}:C\to C,\quad
\mathsf{uses}:C\to M,
\quad
\mathsf{refines}:M\to M.
\]
Suppose the symbolic theory contains the valid inclusion
\[
\mathsf{supports};\mathsf{uses}\subseteq\mathsf{uses}
\]
but deliberately omits
\[
\mathsf{uses};\mathsf{refines}^{-1}\subseteq\mathsf{uses},
\]
which might be unsound depending on the semantics of refinement.

A query is embedded and retrieves claim $c_0$. The certified graph expands from $c_0$ through $\mathsf{supports};\mathsf{uses}$ to method $m$. The result carries a proof tree with two source-backed leaves. Meanwhile, the quiver representation predicts that another method $m'$ has low residual under a learned composite operator, and the sheaf layer finds that inserting $m'$ would lower local energy. The system may therefore rank $m'$ as an excellent \emph{hypothesis} for further evidence search. It may not state that $c_0$ uses $m'$ as a certified fact until a source or symbolic derivation admits it.

This example illustrates the architecture's central asymmetry: continuous mathematics proposes; discrete proof certifies.

\section{What Is New and What Is Not}

\subsection{Established ingredients}
The following are not claimed as novel: quivers and path algebras; representations of quivers; category-theoretic schemas; Horn/Datalog fixed-point semantics; semiring provenance; KGE relation operators; mixed geometry; cellular sheaves; sheaf Laplacians; graph diffusion; or proof-carrying verification patterns. All have substantial prior literatures \citep{spivak2012,green2007,necula1997,gebhart2023,ganev2022,abiteboul1995,hansen2019spectral,jeh2003}.

\subsection{What concurrent work already covers}
Honesty about novelty requires naming the nearest neighbors rather than only the ancestors. Three claims that a reader might expect us to make are \emph{not} available. First, proof-carrying certificates attached to LLM pipeline outputs are not new; \citet{koomullil2026} develop a kernel-checkable version. Second, attributable symbolic execution traces over retrieved evidence are not new; \citet{gann2026} produce them from synthesized Prolog modules and report benchmark gains. Third, using personalized PageRank over a knowledge graph for multi-hop retrieval is not new \citep{gutierrez2024}. What remains, and what Section~\ref{sec:concurrent} argues is the actual gap, is that none of these separates the \emph{grammar of admissible composition} from the \emph{geometry of proposal}, and none states an invariance property over a design space of learned components.

\subsection{Proposed synthesis}
The research contribution is the combined architecture and the sharp division of epistemic responsibilities:
\begin{enumerate}[leftmargin=*]
    \item a symbolic knowledge layer generated by source-attested facts and sound closure;
    \item a quiver-algebraic representation layer in which relation types act as typed operators and path identities become structural constraints;
    \item a metric/mixed-geometry layer used for semantic discovery;
    \item a sheaf layer used for local-to-global consistency;
    \item provenance and proof objects used as an admissibility condition for certified retrieval, and as the index for retraction;
    \item an explicit non-interference theorem (Theorem~\ref{thm:noninterference}) guaranteeing that arbitrary replacement of learned retrieval components cannot introduce an uncertified fact into the certified channel, together with a matching completeness bound (Theorem~\ref{thm:completeness}) that attributes every recall failure to one of three named causes;
    \item a coupling result (Proposition~\ref{prop:section-path-eq}) showing that the sheaf layer's consistency energy is a localized measurement of the representation layer's compositional constraints, so the two layers are not parallel decorations;
    \item a scope theorem (Section~\ref{sec:reptype}) delimiting which representation-theoretic tools are available---essentially every practical schema is of wild type---and minimality results for the retrieval operator itself: a bounded-bisimulation index that exactly decides admissible-word realizability, and a provenance-based characterization of redundant base facts.
\end{enumerate}

We call this design \emph{proof-carrying geometric retrieval}. Its novelty, if validated empirically, is architectural rather than foundational: modern algebraic and geometric machinery is arranged so that statistical power and epistemic authority are intentionally non-coincident.

\section{Limitations}

First, symbolic soundness is only relative to source correctness and rule validity. The framework cannot prove that a corrupted source or mistaken curator is correct, and Section~\ref{sec:extraction} makes clear that in an LLM-ingested corpus this assumption is doing a great deal of work; the attestation error rate is the honest measure of how much. Second, constructing a high-quality typed schema and rule set is costly, and this cost is the framework's main barrier to adoption relative to schema-free GraphRAG pipelines \citep{edge2024,gutierrez2024}. Third, learned representations may add limited value when the graph is tiny or relations are already exhaustive. Fourth, sheaf construction may require learned or curated restriction maps; Definition~\ref{def:induced-sheaf} removes this freedom by tying restrictions to $\rho$, but a system that instead learns them freely reacquires the problem that poor restrictions yield misleading consistency signals. Fifth, recursive provenance can become large and requires compact DAGs, absorptive semantics, or bounded proof policies. Sixth, exact path-algebra constraints may be inappropriate where relations are intrinsically noisy; regularization rather than equality is then preferable.

Seventh, and most importantly for a reader deciding whether to build on this: the paper is entirely theoretical. Every claim about the \emph{benefit} of the architecture---that geometry improves recall, that algebraic regularization helps compositional generalization, that sheaf energy flags incoherence---is a conjecture stated in Section~\ref{sec:hypotheses}, not a result. Only the structural claims (Theorems~\ref{thm:noninterference}, \ref{thm:completeness}, \ref{thm:gauge}, Proposition~\ref{prop:section-path-eq}) are established, and those are established by construction rather than by evidence about the world. The framework should be read as a specification and a set of testable predictions, and it should not be cited as showing that proof-carrying geometric retrieval outperforms anything.

Finally, the framework does not imply that every useful knowledge graph should be represented by a finite-dimensional linear quiver representation. Nonlinear maps, probabilistic kernels, enriched categories, operator algebras, or higher-order relational structures may be better in particular domains. The quiver representation is a disciplined baseline, not a universal theorem about cognition.

\section{Future Directions}

Several mathematical extensions are natural.

\paragraph{Enriched and quantaloid semantics.}
Because $\Rel$ is locally ordered, one can formulate the schema in an order-enriched categorical setting where inclusions are native $2$-cells. Weighted relations suggest enrichment over quantales, potentially unifying symbolic reachability and graded costs.

\paragraph{Nonlinear quiver representations.}
Replace $\Vect$ by smooth manifolds and relation arrows by differentiable maps, or by Markov kernels for probabilistic semantics. Composition remains functorial while permitting nonlinear geometry.

\paragraph{Derived and homological diagnostics.}
The consistency dimension $\sigma_\rho(H)$ (Section~\ref{sec:sheaf-cohomology}) and the Hochschild cohomology sketch of Section~\ref{sec:diagnostics} open a systematic homological program. Computing $HH^*(A_Q)$ for concrete schemas via the bar resolution yields explicit generators that bound the deformation space of any trained representation and predict which regions of parameter space are compatible with the declared path equations. Derived categories $D^b(\mathrm{mod}\text{-}A_Q)$ provide the natural ambient category: $\Ext^n_{A_Q}(M,N)\cong\Hom_{D^b}(M,N[n])$, and $\Ext^1_{A_Q}(M_\rho,M_\rho)$ parameterizes the first-order deformation directions of $M_\rho$, with obstructions to integrating them controlled by $\Ext^2_{A_Q}(M_\rho,M_\rho)$; a nonzero first-order class supplies a concrete direction in operator space that could guide regularization or generate adversarial test cases for compositional generalization. Higher persistent sheaf cohomology, indexing $\sigma_\rho(H)$ across filtrations of the retrieved subgraph by evidence score or hop depth, could track how the space of consistent configurations collapses as context is expanded---a topological analogue of precision-recall curves with a certificate rather than a threshold.

\paragraph{Persistent and topological retrieval.}
Filtrations by score, hop count, or evidence threshold induce filtered complexes. Persistent homology could identify retrieval structures stable across thresholds, complementing the simple top-$k$ margin theorem.

\paragraph{Certified neural-symbolic training.}
One may train relation operators and sheaf restrictions against certified derivations while reserving symbolic proof checking as an independent verifier, analogous to a prover--checker split.

\paragraph{Evidence acquisition agents.}
Hypothesis edges with high geometric score but no proof can trigger tools that search external sources. The resulting evidence passes through the same admission predicate before promotion. This provides a principled route from statistical discovery to certified knowledge growth.

\section{Conclusion}

A grounded retrieval system should be more than a graph whose nodes happen to have embeddings. Geometry, algebra, logic, and provenance answer different questions. When these questions are made explicit, a coherent architecture emerges.

The symbolic layer defines which facts are certified and which compositions are sound. The quiver and path algebra define typed compositional structure. Quiver representations turn that structure into learnable operators while preserving type and composition information. Metric and mixed-geometry spaces support semantic discovery. Cellular sheaves quantify whether local representations are mutually compatible. Semiring provenance and proof objects explain derivations. An LLM sits downstream of these structures and synthesizes their output rather than silently redefining them.

The resulting principle is simple:
\[
\boxed{
\text{geometry proposes}
\;\longrightarrow\;
\text{algebra structures}
\;\longrightarrow\;
\text{proof certifies}
\;\longrightarrow\;
\text{generation communicates}
}
\]

This separation does not reduce the role of learning. It makes learning safer to use aggressively: retrieval and hypothesis generation may exploit sophisticated geometric models precisely because the architecture prevents latent similarity from being mistaken for ground truth. We suggest that this is a productive mathematical foundation for the next generation of graph-based retrieval systems.

\bibliographystyle{plainnat}
\bibliography{pagr}

\appendix
\section{Notation Summary}
\begin{longtable}{p{0.2\textwidth}p{0.7\textwidth}}
\toprule
Symbol & Meaning \\
\midrule
$Q$ & finite schema quiver \\
$Q_0,Q_1$ & entity sorts and primitive relation types \\
$kQ$ & path algebra of $Q$ over field $k$ \\
$J$ & ideal generated by path equations \\
$\A_Q=kQ/J$ & quotient path algebra (see Remark~\ref{rem:admissible}) \\
$\mathcal{E}_J$ & declared path-equation pairs generating $J$ \\
$P=\bigsqcup_i P_i$ & typed entity universe; $P_i$ are entities of sort $i$ \\
$E$ & source-attested extensional relation base \\
$\T$ & positive Horn relational rule set \\
$\Cl_{\T}(E)$ & least certified closure \\
$\At(\mathfrak{M})$ & atoms satisfied by the intended model $\mathfrak{M}$ \\
$\M_i$ & type-specific metric/geometric space \\
$\phi_i,\psi_i$ & entity and query encoders into $\M_i$ \\
$V_i$ & vector space assigned to sort $i$ in quiver representation \\
$A_a$ & linear operator assigned to relation arrow $a$ \\
$G_{\bm d},\Iso(\bm d)$ & base-change group and its isometry subgroup \\
$\F,\F^{\rho}$ & cellular sheaf; the sheaf induced by $\rho$ \\
$C^0(G;\F),C^1(G;\F)$ & vertex-stalk and edge-stalk cochain groups \\
$L_{\F}$ & sheaf Laplacian \\
$H^0(G;\F)$ & zeroth sheaf cohomology (global sections, $=\ker\delta$) \\
$H^1(G;\F)$ & first sheaf cohomology ($=C^1/\operatorname{im}(\delta)$; a topological residual, not a consistency diagnostic---see Remark~\ref{rem:h1-not-obstruction}) \\
$\sigma_\rho(H)$ & consistency dimension $\dim H^0(H;\F^{\rho}|_H)$ of retrieved subgraph $H$ \\
$HH^n(A_Q)$ & $n$-th Hochschild cohomology of the schema algebra $A_Q$ \\
$\Prov(f)$ & provenance object for fact $f$ \\
$\varphi$ & scalarization of a provenance object for ranking \\
$G^{\star}$ & certified internal knowledge set \\
$\theta,\Theta$ & learned configuration and its design space \\
$W$ & column-stochastic diffusion matrix (Stage V) \\
$h,k$ & hop bound and seed count \\
\bottomrule
\end{longtable}

\section{Minimal Machine-Checkable Invariants}
\label{app:invariants}
A practical implementation should enforce at least the following invariants. Invariants 1--5 are together sufficient for Theorem~\ref{thm:noninterference}; the remainder are hygiene conditions that make the paper's diagnostics meaningful.
\begin{enumerate}[leftmargin=*]
    \item Every base fact is well typed under $Q$ and resolves to an evidence object admitted by a policy independent of the extractor that proposed it (Design Principle~\ref{prin:attestation}).
    \item Every derived certified fact has a valid derivation DAG whose internal nodes instantiate rules in $\T$.
    \item Path-equation constraints in $J$ are distinguished from Horn inclusions in $\T$.
    \item Learned edges are stored in a hypothesis relation namespace separate from certified relations.
    \item Retrieval output preserves epistemic status through ranking and serialization.
    \item\label{prin:gauge-log} Any relation-representation score used for ranking is invariant to the declared gauge group, and the declared group is recorded; per Remark~\ref{rem:gauge-gap} it is $\Iso(\bm d)$ and not $G_{\bm d}$ unless operators are orthogonally parameterized.
    \item Cross-type seed scores are calibrated to a common scale before top-$k$ selection (Definition~\ref{def:calibrated}).
    \item The query seed margin is logged when robustness claims depend on top-$k$ stability.
    \item Recall failures are attributed to seed, horizon, or grammar misses (Corollary~\ref{cor:attributable}).
    \item Sheaf energy is reported as a consistency statistic, never as factual confidence.
    \item Retraction of a base fact triggers provenance-indexed invalidation of every certificate depending on it (Section~\ref{sec:retraction}), and evidential multiplicity is reported rather than minimized away (Design Principle~\ref{prin:redundancy}).
    \item Any bisimulation index used to accelerate Stage~II is treated as a sound over-approximation and its results validated against the certified graph (Remark~\ref{rem:index-caveat}).
    \item Generator citations are checked for containment in the supplied context, while claim-level source faithfulness is evaluated separately.
\end{enumerate}

\section{Alternative Categorical Formulation}
One can package parts of the framework categorically. Path equations generate a category $\mathcal{C}=\Path(Q)/\sim_J$. A strict linear representation is a functor
\[
\rho:\mathcal{C}\to\Vect_k.
\]
A relational interpretation is a functor into $\Rel$ only for equational constraints. Horn inclusions require the local order on $\Rel$ to be used explicitly: a rule $p\subseteq q$ is interpreted as a $2$-cell
\[
I(p)\Rightarrow I(q)
\]
in the order-enriched sense. Thus the most faithful categorical home for the combined schema is not an ordinary $1$-category alone but an order-enriched or locally posetal setting. This explains why a plain ``finitely presented category'' captures path identities but not all one-way inference rules.

\section{Relationship Between Quiver Representations and Sheaves}
Both constructions attach vector spaces and maps to combinatorial data, but they serve different roles here. A quiver representation attaches spaces to schema vertices and maps to \emph{relation types}; many concrete entity edges share the same typed operator. A cellular sheaf attaches spaces and restrictions to \emph{individual graph cells or incidences}; different entity edges may impose different compatibility maps. The former is therefore a global typed parameterization, while the latter is a local compatibility system.

Definition~\ref{def:induced-sheaf} exhibits one canonical coupling: the sheaf $\F^{\rho}$ in which every edge inherits its restriction from the relation operator of its label. Under that tying the two structures carry the same information, and Proposition~\ref{prop:section-path-eq} shows the consequence---sections of $\F^{\rho}$ satisfy the declared path equations on their support. The general case is strictly larger. A free cellular sheaf can assign different restrictions to two edges bearing the same relation label, which lets it model edge-specific context that no representation of $Q$ can express, at the cost of losing the guarantee that consistency implies compositional coherence. The design choice is therefore a real one: tying restrictions buys Proposition~\ref{prop:section-path-eq}, freeing them buys expressiveness. A hybrid in which $\F_{u\unlhd e}=A_{\lambda(e)}+\Delta_e$ with $\|\Delta_e\|$ penalized interpolates between the two, and the penalty weight is the knob controlling how much the guarantee degrades. We flag this as the most concrete open modeling question raised by the framework.

\end{document}